\documentclass[11pt,reqno]{amsart}
\usepackage[utf8]{inputenc}
\usepackage{amsthm,amsmath,amssymb,bm}
\usepackage[margin=1.15in]{geometry}
\usepackage{tikz}
\usetikzlibrary{decorations.pathreplacing,calc}
\usepackage{mathtools}
\usepackage{float}
\usepackage[colorlinks=true,linkcolor=blue,citecolor=blue,urlcolor=blue]{hyperref}

\newtheorem{theorem}{Theorem}[section]
\newtheorem{lemma}[theorem]{Lemma}
\newtheorem{proposition}[theorem]{Proposition}
\newtheorem{corollary}[theorem]{Corollary}

\newtheorem{remark}[theorem]{Remark}
\newtheorem{example}[theorem]{Example}
\newtheorem{conjecture}[theorem]{Conjecture}

\newtheorem*{thmA}{Theorem A}
\newtheorem*{thmB}{Theorem B}
\newtheorem*{thmC}{Theorem C}

\newcommand{\R}{\mathbb{R}}
\newcommand{\lmax}{\lambda_{\max}}
\newcommand{\RG}{R_G}
\newcommand{\Nb}[1]{N[#1]}
\newcommand{\dx}{d_x}
\newcommand{\dy}{d_y}

\begin{document}

\title[Discrete Einstein metrics on unicyclic graphs]{Discrete Einstein metrics on
unicyclic graphs}

\author{Shuliang Bai}
\thanks{Beijing Yanqi Lake Institute of Mathematical Sciences and Applications, Beijing, China.}

\date{\today}

\begin{abstract}
In earlier work with Cheng and Hua we showed that on a finite tree the
Lin--Lu--Yau discrete Einstein equation is a single linear eigenvalue problem:
the metric is the Perron eigenvector of an edge-indexed Ricci matrix. We extend
this theory to unicyclic graphs. For geodesic cycle weightings of girth at least
four, we characterize exactly when the same local matrix computes the true
curvature; this is the balanced regime. Schur complementation of the pendant
trees then reduces the eigenvalue problem to a periodic transfer equation on
the cycle. We obtain closed formulas for bare cycles and regular suns, prove a
sharp maximum-degree bound, and show that the spectral top of a fixed localized
defect on a growing cycle converges to the top bound state of the corresponding
infinite-line operator. This yields an explicit golden-ratio threshold for
balance. Outside the balanced regime the active shortest paths and optimal
transport constraints must also be tracked, and phenomena impossible on trees
occur: the bare four-cycle has a one-parameter
family of Einstein metrics, whereas a triangle with one pendant leaf has none.
For every regular sun we construct the symmetric Einstein metric and prove that
it is the unique Einstein metric, up to scale, among all positive weightings.
As long as a normalized Ricci-flow trajectory remains in the balanced cone, its
projective dynamics is generated by the fixed Ricci matrix.
\end{abstract}


\maketitle

\section{Introduction}\label{sec:intro}

Ollivier \cite{Ollivier} and Lin--Lu--Yau \cite{LLY} defined Ricci curvature on graphs
through optimal transport. The resulting edge curvature links local transport geometry to
global and combinatorial properties of a graph. Positive lower curvature bounds give spectral-gap
and concentration estimates for Markov chains \cite{Ollivier}; on graphs, curvature has been
related to the normalized Laplacian spectrum \cite{BauerJostLiu}, local clustering
\cite{JostLiu}, volume and diameter \cite{Paeng}, and Bonnet--Myers rigidity
\cite{CushingKamtueKoolenLiuMunchPeyerimhoff}. Further developments include general graph
Laplacians, heat-flow and diameter estimates \cite{MW}, the dependence on the idleness parameter
\cite{BCLMP}, and parallel Bakry--\'Emery curvature and eigenvalue estimates
\cite{LinYau,CLP}. Weighted and large-scale directions include curvature from effective
resistance \cite{DevriendtLambiotte}, sharp bounds for sums of weighted curvatures
\cite{bai2021ricci}, and convergence to manifold Ricci curvature in random geometric graphs
\cite{vanderHoorn}. In network science, Ollivier--Ricci curvature has been used to study
Internet topology \cite{NiLGGS}, community structure \cite{NiLinLuoGao}, and multiscale
bottlenecks and information spreading \cite{GosztolaiArnaudon}.

Write $\kappa_e(w)$ for the Lin--Lu--Yau (LLY) curvature of an edge $e$ under a positive
weighting $w:E\to\R_{>0}$ (defined in Section~\ref{sec:prelim}). We call $w$ a
discrete Einstein metric, or say that it has constant curvature, if there is
a real number $\kappa$ with
\[
\kappa_e(w)=\kappa\qquad\text{for every edge }e\in E .
\]
Every edge carries a strictly positive weight, so this is a single condition imposed on all
edges simultaneously. It is the discrete analogue of the Einstein condition
$\mathrm{Ric}(g)=\kappa g$. For the uniform-neighbour measures \eqref{eq:measure} and the
weighted path metric used throughout this paper, it is precisely, up to overall scaling, the
fixed-point condition of the normalized edge-length flow
$\dot w_e=-(\kappa_e-\bar\kappa)w_e$, where $\bar\kappa$ is the weighted mean
curvature; see Remark~\ref{rem:flow}. This statement is
convention-specific: related continuous- and discrete-time graph Ricci flows use several choices
of random-walk measure, normalization, and metric update
\cite{BLLWY,LaiBaiLin,LiMunch,NiLinGaoGu,NiLinLuoGao}.
On unweighted graphs, where $w\equiv1$ is fixed, constant-curvature
examples are classical: the cycles $C_g$ $(g\ge6)$ are Ricci-flat, and the Ricci-flat graphs
of maximum degree at most four have been completely classified \cite{LLY,LLYflat,BLYflat};
here, by
contrast, the weighting itself is the unknown, and the question is which weightings render the
curvature constant.

For finite trees, in joint work with Cheng and Hua \cite{BCH}, we showed that this
problem is completely solvable and linear. We introduced the Ricci matrix
$R_T\in\R^{E\times E}$,
\begin{equation}\label{eq:RT}
(R_T)_{e,e'}=
\begin{cases}
-\bigl(\tfrac1{\dx}+\tfrac1{\dy}\bigr), & e=e'=\{x,y\},\\[2pt]
\tfrac1{d_z}, & e\neq e',\ e\cap e'=\{z\},\\[2pt]
0, & e\cap e'=\varnothing,
\end{cases}
\end{equation}
proved that a positive weight $w$ is discrete Einstein of curvature $\kappa$ if and only if
\[
   R_Tw=\lmax(R_T)\,w,
   \qquad \kappa=-\lmax(R_T),
\]
and deduced existence and uniqueness from Perron--Frobenius theory, together with sharp bounds
and a caterpillar classification \cite{BCH,ChengHua,BCHsub,BCHleaf}. The reason is that on a
tree $\kappa_{xy}\,w_{xy}$ is linear in $w$, which is exactly what makes $R_T$ a fixed matrix.

The curvature--spectrum relation here is therefore different from the inequalities connecting
curvature bounds to the vertex Laplacian spectrum of a fixed graph. The metric itself is the
unknown vector on $E$, and on a tree the complete curvature equation is the exact edge-space
identity
\[
   \bigl(\kappa_e(w)w_e\bigr)_{e\in E}=-R_Tw.
\]
For a unicyclic graph $G$, let $R_G$ be given by the same local formula \eqref{eq:RT}. The
central innovation of this paper is to determine exactly when the corresponding identity with
$R_G$ survives the addition of a cycle, to exploit the resulting periodic matrix spectrum, and
to describe what replaces it when shortest paths and optimal-transport constraints change.
Thus the Ricci matrix does not merely bound curvature: on its precise domain of validity it
encodes the full edge-weighted curvature vector, and the Einstein equation becomes its Perron
eigenvalue problem.

We carry this programme out on unicyclic graphs (connected graphs with $|E|=|V|$), the first
non-tree class in the Ricci-matrix framework.

Throughout the paper we distinguish carefully between the balanced regime, where the
local tree-formula Ricci matrix computes the true curvature, and the general transport
problem, where shortest-path and optimal-transport cells may refine the balanced/wrap
pattern. For regular suns the equal endpoint degrees reduce the geodesic problem to the
two exact regimes analysed below, and a metric-closure argument excludes non-geodesic
Einstein metrics. Only the proposed extension of the localized-defect threshold to general
multi-vertex decorations remains conjectural.

\subsection*{Results} Fix a unicyclic graph $G$ with cycle
$C=(v_0,v_1,\dots,v_{g-1})$, cycle edges $c_1,\dots,c_g$ (with $c_j=\{v_{j-1},v_j\}$) and
perimeter $L=\sum_j w(c_j)$. Define $\RG$ by the same local formula \eqref{eq:RT}. Since
$G$ is connected, its line graph is connected, so $\RG+2I$ is irreducible nonnegative and
Perron--Frobenius gives a simple $\lmax(\RG)$ with positive eigenvector $w^\ast$
(Lemma~\ref{lem:pf}).

Our first result determines, within the geodesic cycle cone and for girth at least four,
exactly when the local matrix $\RG$ computes the curvature. In that case its Perron vector
has Einstein curvature $-\lmax(\RG)$.

\begin{thmA}
Assume that the cycle has girth $g\ge4$ and that its weighting is geodesic, i.e.\ every
cycle edge has weight at most $L/2$. A cycle edge $c_j$ obeys the tree formula \eqref{eq:tree-kappa} if and only if its
three consecutive cycle edges satisfy
$w(c_{j-1})+w(c_j)+w(c_{j+1})\le L/2$, and every off-cycle edge
(equivalently, every edge in an attached pendant tree) obeys the tree formula.
Consequently, if the Perron eigenvector $w^\ast$ of $\RG$ is geodesic, then it is a
discrete Einstein metric with constant curvature $\kappa=-\lmax(\RG)$ if and only if this
inequality holds at every cycle edge. Moreover every balanced Einstein metric is
proportional to $w^\ast$; thus uniqueness here is uniqueness within the balanced regime.
\end{thmA}

\noindent We call the length inequality the balanced regime. It is proved in
Proposition~\ref{prop:balanced} and Theorem~\ref{thm:einstein}; for a regular decorated
cycle it holds automatically once $g\ge6$ (Corollary~\ref{cor:sym}). In the balanced regime
the eigenproblem of $\RG$ reduces to a periodic transfer equation on the cycle.
Lemma~\ref{lem:feedback} Schur-eliminates each pendant tree to a scalar feedback term, and
Theorem~\ref{thm:reduction} characterizes the remaining eigenvalues by the trace-$2$
condition for the ordered product of the $2\times2$ transfer matrices. This is the periodic
(Bloch) analogue of the Dirichlet--Chebyshev equation in the tree theory
\cite{BCH,BCHsub}.

The reduction gives explicit comparison families. Proposition~\ref{prop:bare} gives the
bare-cycle spectrum $\{\cos(2\pi k/g)-1:k=0,\dots,g-1\}$, while
Theorem~\ref{thm:sun} gives
\[
 \lmax(R_{S_{g,d}})=\frac{2\bigl(\sqrt{d-1}-1\bigr)}{d}
\]
for every regular sun, independently of $g$. For $g\ge6$ its Perron metric is balanced and
has curvature tending to $0$ from below as $d\to\infty$; for $g\le5$ the same spectral
formula need not compute the Einstein curvature. Theorem~\ref{thm:degbound} proves the same
universal maximum-degree bound as for trees, using $|E|=|V|$ to orient every vertex to a
unique parent.

Beyond the balanced regime the problem changes character. If a cycle edge is non-geodesic,
its raw weight $w_e$ differs from its metric length $d_w(x,y)$, and
$\kappa_e(w)w_e$ need not even be piecewise linear. Moreover, on a general unicyclic graph
the set of edges violating the balanced condition does not by itself determine a transport
cell: further shortest-path and optimal-transport constraints may change while that set stays
fixed. On any fully specified geodesic cell where the curvature map is linear, the Einstein
equation is still a self-consistent eigenvalue problem. For regular suns the equal endpoint
degrees make the balanced/wrap dichotomy sufficient in the geodesic cone;
Section~\ref{sec:exist} then excludes non-geodesic Einstein metrics by metric closure. The
general unicyclic case remains open.

Section~\ref{sec:short} computes the active operators for $C_3,C_4,C_5$ in the geodesic
cell containing the uniform weighting (Proposition~\ref{prop:short}). The $C_4$ operator is
reducible and yields a one-parameter family of positive Einstein metrics up to scaling.
Existence can also fail: the triangle with one pendant leaf has an edge of curvature
$\tfrac32$, while its leaf edge has curvature strictly below $\tfrac43$, so it admits no
Einstein metric (Example~\ref{thm:nonexist}). By contrast, symmetry gives an Einstein metric
on every bare cycle and every regular sun
(Example~\ref{ex:barecycles}, Corollary~\ref{cor:symexist}). For regular suns this existence
statement extends to unconditional uniqueness.

\begin{thmB}
For every $g\ge3$ and $d\ge3$, every regular sun $S_{g,d}$ --- a cycle of girth $g$
with $d-2$ pendant leaves attached at
each cycle vertex --- carries exactly one discrete Einstein metric up to scaling among all
positive weightings: the symmetric metric of Proposition~\ref{prop:sununiq}.
\end{thmB}

\noindent The proof combines symmetric existence (Corollary~\ref{cor:symexist}),
uniqueness throughout the geodesic cone (Proposition~\ref{prop:sunglobal}), and the exclusion
of a non-geodesic cycle edge (Proposition~\ref{prop:nongeo}); see
Theorem~\ref{thm:sununcond}. It is the first unicyclic family in this paper for which both
existence and uniqueness are settled beyond the balanced regime.

A final result sharpens the balanced regime for a single localized defect --- a cycle
carrying one pendant tree at a single vertex --- into an exact threshold.

\begin{thmC}
Let $G_g$ be a cycle $C_g$ carrying one fixed nonempty rooted pendant tree at a single
cycle vertex, and let $\lambda_g=\lmax(R_{G_g})$. Then
$\lambda_g\to\lambda_\infty>0$, where $\lambda_\infty$ is the spectral top of the
corresponding infinite-line defect operator. If $\mu>1$ is defined by
$\mu+\mu^{-1}=2(1+\lambda_\infty)$, then the limiting balance ratio is
\[
 g_{\!*}(\mu)=\frac{(2\mu+1)(\mu-1)}{2\mu^2}.
\]
Consequently, if
\[
\lambda_\infty<\varphi-\tfrac32=\tfrac{\sqrt5-2}{2},
\qquad \varphi=\tfrac{1+\sqrt5}{2},
\]
then the Perron metric is balanced, and hence Einstein, for all sufficiently large girth;
if the inequality is reversed, it is unbalanced for all sufficiently large girth. The
equality case is not decided by the limiting ratio alone.
\end{thmC}

\noindent Theorem~C treats the simplest nontrivial localized defect: one fixed pendant tree
attached at a single vertex of a growing cycle. The bound-state convergence is proved in
Theorem~\ref{thm:bound-state}, and Proposition~\ref{prop:threshold} converts it into the
stated golden-ratio criterion. Above the threshold $w^\ast$ is unbalanced, so $\RG$ no
longer computes an Einstein metric; existence in that transport region is a separate question
(Remark~\ref{rem:threshold-exist}). The global Perron spread $\rho_G$ of
Proposition~\ref{prop:apriori} is not the operative quantity --- it grows like $\mu^{g/2}$
(Remark~\ref{rem:rho-unbounded}) --- whereas the local spectral strength
$\lambda_\infty$ remains finite. For a general localized multi-vertex decoration we
conjecture the same threshold to be sufficient (Conjecture~\ref{conj:general-threshold}).

\subsection*{Organization}
Section~\ref{sec:prelim} fixes notation and reviews the Lin--Lu--Yau curvature and the local
Ricci matrix. Section~\ref{sec:balanced} establishes the balanced regime and proves
Theorem~A. Section~\ref{sec:reduction} develops the periodic reduction.
Section~\ref{sec:closed} derives the closed forms and the localized-defect asymptotics of
Theorem~C. Section~\ref{sec:short} treats short cycles. Section~\ref{sec:exist} proves the
non-existence example and the regular-sun uniqueness result, Theorem~B.

\section{Preliminaries}\label{sec:prelim}

Let $G=(V,E,w)$ be a finite connected graph with edge weights $w:E\to\R_{>0}$; write
$w_{xy}=w(\{x,y\})$, let $d_v$ be the combinatorial degree of $v$, and set
$S_v=\sum_{u\sim v}w_{vu}$.

We write $d(x,y)$ for the \emph{weighted distance}, the least total weight of a path
joining $x$ and $y$,
\[
d(x,y)=\min_{P:\,x\rightsquigarrow y}\ \sum_{e\in P}w_e .
\]
For adjacent $x,y$ this is $\min\bigl(w_{xy},\ \text{length of any other route}\bigr)$; on
a tree $d(x,y)=w_{xy}$, but on a graph with a short cycle a route around the cycle
can be shorter; this is the source of the phenomena below.

For an idleness $\alpha\in[0,1]$ we use the lazy one-step random-walk measure at $x$, which
keeps mass $\alpha$ at $x$ and spreads the rest uniformly over the neighbours,
\begin{equation}\label{eq:measure}
\mu_x^\alpha(z)=
\begin{cases}
\alpha, & z=x,\\[2pt]
\dfrac{1-\alpha}{d_x}, & z\sim x,\\[4pt]
0, & \text{otherwise.}
\end{cases}
\end{equation}

For two probability measures $\mu,\nu$ on $V$ the ($L^1$-)Wasserstein distance is, by
Kantorovich duality,
\begin{equation}\label{eq:wasserstein}
W(\mu,\nu)=\min_{\pi}\ \sum_{a,b}\pi(a,b)\,d(a,b)
=\max_{\substack{f:V\to\R\\ 1\text{-Lipschitz}}}\ \sum_{z}f(z)\,\bigl(\mu-\nu\bigr)(z),
\end{equation}
the minimum over couplings $\pi\ge0$ with marginals $\mu,\nu$ (transport plans), and the
maximum over $f$ with $|f(a)-f(b)|\le d(a,b)$ for all $a,b$ (Kantorovich potentials).

The $\alpha$-Ollivier curvature of an edge
$x\sim y$ is $\kappa_\alpha(x,y)=1-W(\mu_x^\alpha,\mu_y^\alpha)/d(x,y)$, and the
Lin--Lu--Yau curvature is the limiting slope as $\alpha\to1$,
\begin{equation}\label{eq:lly}
\kappa_{xy}=\lim_{\alpha\to1}\frac{\kappa_\alpha(x,y)}{1-\alpha}
\end{equation}
(the ratio is constant on an interval $[\alpha_0,1]$, so the limit exists \cite{LLY}). On
a tree it has the closed form
\begin{equation}\label{eq:tree-kappa}
\kappa_{xy}=-\Big(\frac{S_x-2w_{xy}}{w_{xy}\dx}+\frac{S_y-2w_{xy}}{w_{xy}\dy}\Big),
\end{equation}
and multiplying by $w_{xy}$ makes the right-hand side linear in $w$, which yields
the Ricci matrix \eqref{eq:RT} \cite{BCH}. We call \eqref{eq:tree-kappa} the tree
formula.

A connected graph is unicyclic if $|E|=|V|$; equivalently, it contains exactly one
cycle. Concretely (Figure~\ref{fig:unicyclic}), a unicyclic graph is a single cycle
$C=(v_0,\dots,v_{g-1})$ of girth $g\ge3$ together with arbitrary rooted trees
attached at the cycle vertices --- at a given vertex the attached tree may be empty, a
single pendant leaf, or a larger branch. Deleting the $g$ cycle edges leaves exactly
these pendant trees, each rooted at a cycle vertex, and $C$ is the $2$-core of
$G$. We write $R_G$ for the fixed local matrix \eqref{eq:RT} on $E(G)$, and
$\lmax:=\lmax(R_G)$. We reserve this notation for the tree-formula matrix throughout.
Whenever all relevant edges are geodesic and the full shortest-path and optimal-transport
data define a cell $\sigma$ on which the curvature map is linear, we write
$R_G^{(\sigma)}$ for the corresponding regime operator, defined by
\[
   \bigl(\kappa_e(w)w_e\bigr)_{e\in E(G)}=-R_G^{(\sigma)}w .
\]
In the balanced regime $R_G^{(\mathrm{bal})}=R_G$; outside it the active operator
$R_G^{(\sigma)}$ need not coincide with $R_G$.

\begin{remark}[Scope of transport regimes]\label{rem:transport-scope}
Neither geodesicity nor the balanced/wrap labels alone give a complete linear-cell
decomposition for arbitrary unicyclic graphs. First, on the four-cycle with cyclic edge
weights $(4,1,t,1)$, $0<t<2$, the first edge is non-geodesic, with metric length $t+2$, and
direct transport gives
\[
 \kappa_{01}=\frac{t+3}{t+2},\qquad
 \kappa_{01}w_{01}=4+\frac4{t+2}.
\]
Indeed, after common-mass cancellation the plan sends $1-3\varepsilon/2$ from $0$ to
$1$ and $\varepsilon/2$ from $3$ to $2$; the distance-to-$1$ potential makes both moves
binding.
Thus multiplication by the raw edge weight destroys linearity; the metric-weighted quantity
$\kappa_{01}d_w(0,1)=t+3$ is the natural one. Second, attach a unit leaf at vertex $1$ of a
four-cycle with cyclic weights $(1,1,1,t)$, $0<t<3$. Every cycle edge is strictly geodesic
and every three-edge window is of wrap type, yet
\[
 \kappa_{01}=\begin{cases}(3-t)/6,&0<t\le1,\\[2pt]1/3,&1\le t<3,\end{cases}
\]
because the competing distance $d_w(3,1)=\min\{t+1,2\}$ changes at $t=1$. In the dual,
normalising $f(0)=1$, $f(1)=0$, $f(3)=u$, and $f(2)=v$ reduces the nontrivial step to
maximising $u/2-v/3$ under $|u-1|\le t$, $|v|\le1$, and $|u-v|\le1$, whose value is
$1/3+\min\{1+t,2\}/6$.
Hence a general
transport cell must record the active shortest paths and transport constraints, not only the
set of balanced/wrap edges. The equal-degree sun formulas of Lemma~\ref{lem:fmax} are a
special reduction in which the two labels do suffice.
\end{remark}

\begin{figure}[t]
\centering
\begin{tikzpicture}[scale=1.15,
  cv/.style={circle,draw,fill=black,inner sep=1.7pt},
  lf/.style={circle,draw,fill=white,inner sep=1.4pt}]
 \foreach \i in {0,...,5} {\coordinate (v\i) at ({90+60*\i}:1.5);}
 \foreach \i in {0,...,5} {\node[cv] (n\i) at (v\i) {};}
 \draw[very thick] (n0)--(n1)--(n2)--(n3)--(n4)--(n5)--(n0);
 \node at (0,0) {\small cycle $C$};
 \node[lf] (a1) at ($(n1)+(150:0.85)$) {}; \draw (n1)--(a1);
 \node[lf] (b1) at ($(n2)+(215:0.85)$) {}; \node[lf] (b2) at ($(n2)+(160:0.85)$) {};
 \draw (n2)--(b1); \draw (n2)--(b2);
 \node[lf] (c1) at ($(n3)+(255:0.85)$) {}; \draw (n3)--(c1);
 \node[lf] (c2) at ($(c1)+(220:0.75)$) {}; \node[lf] (c3) at ($(c1)+(300:0.75)$) {};
 \draw (c1)--(c2); \draw (c1)--(c3);
 \node[lf] (d1) at ($(n4)+(300:0.85)$) {}; \draw (n4)--(d1);
\end{tikzpicture}
\caption{A unicyclic graph: one cycle $C$ (bold edges, filled vertices) with rooted trees
attached at its vertices (white vertices) --- here none at two vertices, a single leaf at
two, two leaves at one, and a depth-$2$ branch at one. Deleting the cycle edges leaves
these pendant trees.}
\label{fig:unicyclic}
\end{figure}
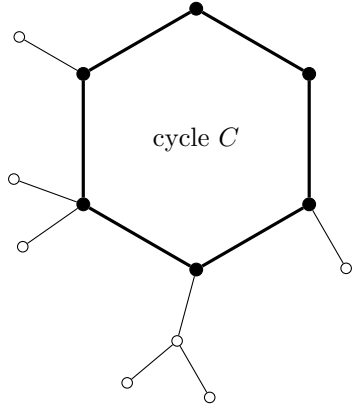

\begin{lemma}\label{lem:pf}
Let $G$ be connected with $|E|\ge2$. Then $\RG+2I$ is entrywise nonnegative and
irreducible; hence $\lmax(\RG)$ is a simple eigenvalue with a strictly positive
eigenvector $w^\ast$, unique up to scaling, and it is the only eigenvalue with a positive
eigenvector.
\end{lemma}
\begin{proof}
Each diagonal entry of $\RG$ is $-(1/\dx+1/\dy)\ge-2$, so $\RG+2I\ge0$; off-diagonal
entries are $1/d_z\ge0$. Two edges have a nonzero off-diagonal entry iff they are
adjacent in the line graph $L(G)$. Since $G$ is connected, $L(G)$ is connected, so
$\RG+2I$ is irreducible. Perron--Frobenius applies; shifting by $-2I$ preserves
eigenvectors.
\end{proof}

The one input we need about $\kappa$ is its locality, immediate from the definition.

\begin{lemma}\label{lem:local}
For an edge $\{x,y\}$, the curvature $\kappa_{xy}$ is a function only of $\dx,\dy$ and of
the matrix of pairwise distances $\big(d(a,b)\big)_{a\in\Nb{x},\,b\in\Nb{y}}$, where
$\Nb{x}=\{x\}\cup\{z:z\sim x\}$.
\end{lemma}
\begin{proof}
$\mu_x^\alpha,\mu_y^\alpha$ are supported on $\Nb{x},\Nb{y}$ with masses determined by
$\alpha,\dx,\dy$. The transport cost $W(\mu_x^\alpha,\mu_y^\alpha)$ is the optimum of a
linear program whose only data are these masses and the costs $d(a,b)$ for
$a\in\Nb{x},b\in\Nb{y}$; and $d(x,y)$ is one of these. Thus $\kappa_\alpha(x,y)$, and its
$\alpha\to1$ limit $\kappa_{xy}$, depend only on the stated data.
\end{proof}

\begin{lemma}\label{lem:mono}
Fix an edge $\{x,y\}$ and, by Lemma~\ref{lem:local}, regard $\kappa_{xy}$ as a function of
the whole distance matrix $D=\big(d(a,b)\big)_{a\in\Nb x,\,b\in\Nb y}$. Let $D'$ be a second
such matrix with
\[
D'(a,b)\le D(a,b)\ \ \text{for all pairs }(a,b),\qquad D'(x,y)=D(x,y),
\]
and write $W,W'$ and $\kappa_\alpha,\kappa'_\alpha$ for the corresponding transport costs and
$\alpha$-curvatures. Then $W'\le W$ for every $\alpha$, hence $\kappa'_\alpha\ge\kappa_\alpha$
and $\kappa'_{xy}\ge\kappa_{xy}$: shortening any neighbour distances, with $d(x,y)$ fixed,
cannot lower the curvature. Quantitatively, if some $W$-optimal plan $\pi^\ast$ places mass
$m(\alpha)$ on a pair $(a_0,b_0)$ with $D'(a_0,b_0)<D(a_0,b_0)$, then
\begin{equation}\label{eq:mono-quant}
\kappa'_\alpha-\kappa_\alpha\ \ge\ \frac{m(\alpha)\,\bigl(D(a_0,b_0)-D'(a_0,b_0)\bigr)}{d(x,y)}\,;
\end{equation}
in particular $\kappa'_{xy}>\kappa_{xy}$ whenever $\liminf_{\alpha\to1}m(\alpha)/(1-\alpha)>0$.
\end{lemma}
\begin{proof}
The comparison uses all pairwise distances at once, through the cost of a transport plan.
For fixed marginals $\mu_x^\alpha,\mu_y^\alpha$ the coupling
polytope $\Pi=\{\pi\ge0:\ \pi\ \text{has marginals}\ \mu_x^\alpha,\mu_y^\alpha\}$ depends only
on the masses, not on the costs. For every $\pi\in\Pi$, since $\pi\ge0$ and $D'\le D$
entrywise,
\[
\langle\pi,D'\rangle=\sum_{a\in\Nb x,\,b\in\Nb y}\pi(a,b)\,D'(a,b)
\ \le\ \sum_{a,b}\pi(a,b)\,D(a,b)=\langle\pi,D\rangle .
\]
Minimising over the common polytope $\Pi$ gives
$W'=\min_{\pi}\langle\pi,D'\rangle\le\min_{\pi}\langle\pi,D\rangle=W$. For the quantitative
statement, let $\pi^\ast$ be $W$-optimal; evaluating the same plan against $D'$ and
keeping only the $(a_0,b_0)$ term,
\[
W'\le\langle\pi^\ast,D'\rangle
=W-\sum_{a,b}\pi^\ast(a,b)\bigl(D(a,b)-D'(a,b)\bigr)
\le W-m(\alpha)\bigl(D(a_0,b_0)-D'(a_0,b_0)\bigr),
\]
all summands being $\ge0$. Since $D'(x,y)=D(x,y)$, dividing $W-W'$ by the fixed normaliser
$d(x,y)$ yields $\kappa'_\alpha-\kappa_\alpha=(W-W')/d(x,y)$ and hence
\eqref{eq:mono-quant}. Finally $\kappa_{xy}=\lim_{\alpha\to1}\kappa_\alpha/(1-\alpha)$, so
dividing \eqref{eq:mono-quant} by $1-\alpha$ and passing to the limit gives
$\kappa'_{xy}-\kappa_{xy}\ge \big(\liminf_\alpha \tfrac{m(\alpha)}{1-\alpha}\big)\cdot
(D-D')(a_0,b_0)/d(x,y)$, which is positive under the stated condition.
\end{proof}

\section{The balanced regime}\label{sec:balanced}

On a unicyclic graph the curvature is not in general given by the tree formula. Under a
geodesic cycle weighting, a pendant edge is not affected by the cycle and its curvature is
given by \eqref{eq:tree-kappa}. A cycle edge is different: besides the edge itself, the two
endpoints are joined by a second route running
the other way around the cycle --- along the \emph{complementary arc}, the cycle with the edge
removed --- which can shorten a transport distance and thereby, by Lemma~\ref{lem:mono}, raise
the curvature. Whether the
tree formula still holds
on a cycle edge is governed by a metric (weight) condition, which we state first and only
then read off the curvature. The mechanism is simply that restoring the deleted cycle edge can only lower a distance
$d(a,b)$ by supplying the alternative path around the cycle.

\begin{proposition}\label{prop:balanced}
Let $G$ be a unicyclic graph whose cycle has girth $g\ge4$, and let
$c_j=\{x,y\}$, $x=v_{j-1}$, $y=v_j$, be a cycle edge with cycle
perimeter $L=\sum_i w(c_i)$, and let $x'=v_{j-2}$, $y'=v_{j+1}$ be the two outer
cycle neighbours of $x,y$. Consider the weight condition
\begin{equation}\label{eq:3sum}
w(c_{j-1})+w(c_j)+w(c_{j+1})\ \le\ \frac{L}{2}
\end{equation}
(the three consecutive cycle edges centred at $c_j$ occupy at most half of the cycle).
Then:
\begin{enumerate}
\item[\rm(i)] if \eqref{eq:3sum} holds, the curvature $\kappa_{c_j}$ is given by the tree
formula \eqref{eq:tree-kappa};
\item[\rm(ii)] if \eqref{eq:3sum} fails and $w(c_j)\le L/2$, the complementary arc strictly shortens the
distance $d(x',y')$, and $\kappa_{c_j}$ is strictly larger than the value given by
\eqref{eq:tree-kappa}.
\end{enumerate}
Every off-cycle edge (that is, every edge in an attached pendant tree) obeys the tree
formula \eqref{eq:tree-kappa} whenever the weighting is geodesic --- no cycle edge exceeds
$\tfrac12 L$ --- as it is throughout the non-degenerate setting considered here (a cycle
edge above half the perimeter is degenerate and excluded, cf.\ case (ii) and
Section~\ref{sec:exist}).
\end{proposition}

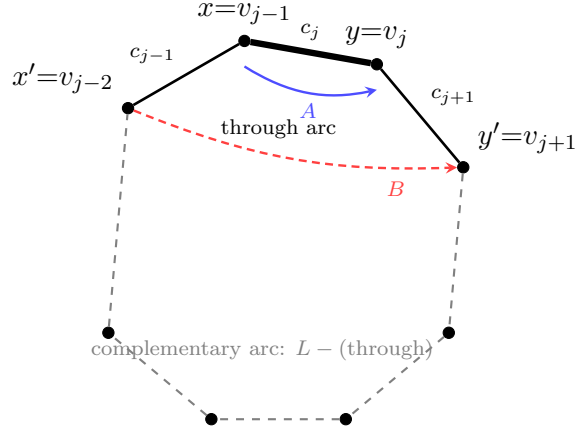
\begin{figure}[t]
\centering
\begin{tikzpicture}[scale=1.3,
  cv/.style={circle,draw,fill=black,inner sep=1.5pt},
  ot/.style={->,>=stealth,line width=0.9pt}]
 \node[cv,label=above left:{$x'{=}v_{j-2}$}] (xp) at (140:2) {};
 \node[cv,label=above:{$x{=}v_{j-1}$}] (x) at (100:2) {};
 \node[cv,label=above:{$y{=}v_{j}$}] (y) at (60:2) {};
 \node[cv,label=above right:{$y'{=}v_{j+1}$}] (yp) at (20:2) {};
 \foreach \a in {330,290,250,210} {\node[cv] (\a) at (\a:2) {};}
 \draw[line width=1pt] (xp)--(x) node[midway,above left,font=\scriptsize]{$c_{j-1}$};
 \draw[line width=2.4pt] (x)--(y) node[midway,above,font=\scriptsize]{$c_j$};
 \draw[line width=1pt] (y)--(yp) node[midway,above right,font=\scriptsize]{$c_{j+1}$};
 \draw[gray,dashed,line width=0.8pt] (yp)--(330)--(290)--(250)--(210)--(xp);
 \draw[ot,blue!70] ($(x)+(0,-0.26)$) to[bend right=22]
   node[below,font=\scriptsize]{$A$} ($(y)+(0,-0.26)$);
 \draw[ot,red!75,densely dashed] (xp) to[bend right=12]
   node[pos=0.83,below=1pt,font=\scriptsize]{$B$} (yp);
 \node[font=\scriptsize] at (90:1.08) {through arc};
 \node[font=\scriptsize,gray] at (262:1.18) {complementary arc: $L-(\text{through})$};
\end{tikzpicture}
\caption{Proposition~\ref{prop:balanced}: the two routes between the outer neighbours
$x',y'$ of a cycle edge $c_j$ --- the through arc (bold) and the complementary arc (dashed).
The tree formula holds at $c_j$ iff the through arc is the geodesic.}
\label{fig:balanced}
\end{figure}

\begin{proof}
Write $\varepsilon=1-\alpha$ and let $T:=w(c_{j-1})+w(c_j)+w(c_{j+1})$ be the length of the
through-arc $x'\!-\!x\!-\!y\!-\!y'$. The two routes between $x'$ and $y'$ have lengths $T$
and $L-T$ (Figure~\ref{fig:balanced}), so
\begin{equation}\label{eq:dxpyp}
d(x',y')=\min(T,\,L-T),
\end{equation}
and the weight condition \eqref{eq:3sum} reads $T\le L/2$. Compare $G$ with the
deleted-arc tree, obtained by deleting any one cycle edge of the complementary arc
(which is nonempty because $g\ge4$). This
keeps $c_{j-1},c_j,c_{j+1}$ and all pendant trees intact, and on a tree the curvature of
$c_j$ is the tree formula \eqref{eq:tree-kappa} \cite{BCH}.

\smallskip\noindent(i) Balanced case $T\le L/2$. Every pair in $\Nb x\times\Nb y$ realises its
distance through the arc $x'xyy'$ --- its cycle portion is at most $T\le L/2$, beating the
complementary $L-T\ge L/2$ --- so on $\Nb x\times\Nb y$ the two distance matrices coincide and,
by locality (Lemma~\ref{lem:local}), $\kappa_{c_j}=\kappa_{\mathrm{tree}}$.

\smallskip\noindent(ii) Unbalanced case $T>L/2$. Restoring the deleted edge can only shorten
distances --- $d_G\le d_{\mathrm{tree}}$ entrywise on $\Nb x\times\Nb y$, the tree being a
subgraph --- and it shortens the outer pair \emph{strictly}, $d_G(x',y')=L-T<T$, while the
normaliser $d(x,y)=w(c_j)$ is unchanged. Here $w(c_j)\le\tfrac12L$, so $c_j$ is a geodesic and
$d(x,y)=w(c_j)$ in both $G$ and the tree; note that the unbalanced condition constrains the
three-edge sum $T$, not the single edge $w(c_j)$. A longer edge $w(c_j)>\tfrac12L$ would be
non-geodesic, with $d(x,y)=L-w(c_j)$: no deleted-arc tree reproduces this distance, so the
tree-edge formula no longer applies and such a weighting is the degenerate case excluded in
Section~\ref{sec:exist}. For $\alpha$ sufficiently close to $1$, $x'$ carries surplus
$\varepsilon/d_x$ and $y'$ carries deficit $\varepsilon/d_y$. In the deleted tree every
transport from the $x$-side to the $y$-side crosses $c_j$, so its cost splits into a source
term, the fixed crossing cost, and a target term. Hence the matching across the two sides may
be rearranged without changing the total cost, and there is a tree-optimal plan that sends
\[
 m_\alpha=\varepsilon\min\{d_x^{-1},d_y^{-1}\}
\]
from $x'$ to $y'$. Running this plan in $G$ lowers that part of the cost by
$m_\alpha(2T-L)$ and raises no other cost. Lemma~\ref{lem:mono} therefore gives the
quantitative strict estimate
\[
 \kappa_{c_j}(G)-\kappa_{c_j}(\mathrm{tree})
 \ge \min\{d_x^{-1},d_y^{-1}\}\,\frac{2T-L}{w(c_j)}>0.
\]
The raised value is computed for a bare short cycle in Proposition~\ref{prop:short}.

\smallskip\noindent Finally, let $e=\{x,y\}$ be any off-cycle edge. Such an edge is a
bridge. Hence every path from a vertex in $\Nb{x}\setminus\{y\}$ to a vertex in
$\Nb{y}\setminus\{x\}$ crosses $e$. All off-cycle incident edges are bridges and therefore
geodesic, while any incident cycle edge is geodesic by hypothesis. Consequently every entry
of the local distance matrix $\bigl(d(a,b)\bigr)_{a\in\Nb{x},\,b\in\Nb{y}}$ is the same sum
of incident edge lengths as in a tree with the same two endpoint stars; in particular
$d(x,y)=w_{xy}$. By locality (Lemma~\ref{lem:local}), the curvature of $e$ therefore agrees
with the tree value \eqref{eq:tree-kappa}.
\end{proof}

\begin{remark}\label{rem:g3}
For $g=3$ the ``$3$-edge sum'' is the whole perimeter $L>L/2$ and $x',y'$ coincide, so
the criterion degenerates; \eqref{eq:3sum} is a genuine proper-sub-arc condition only for
$g\ge4$ (a strict $3$-edge sub-arc for $g\ge5$). Short cycles are treated separately in
Section~\ref{sec:short}.
\end{remark}

\begin{theorem}\label{thm:einstein}
Let $G$ have cycle girth $g\ge4$, and let $w^\ast>0$ be the Perron eigenvector of $\RG$, with eigenvalue
$\lmax=\lmax(\RG)$, and assume that every cycle edge is geodesic in $w^\ast$.
Then $w^\ast$ induces constant curvature equal to $-\lmax$ on every edge --- i.e.\
$w^\ast$ is a discrete Einstein metric with $\kappa=-\lmax$ --- if and only if, measured
in $w^\ast$, every cycle edge satisfies \eqref{eq:3sum} (no three consecutive cycle edges
exceed half the perimeter). Moreover, every balanced Einstein metric is proportional to
$w^\ast$.
\end{theorem}
\begin{proof}
The eigen-equation $\RG w^\ast=\lmax w^\ast$ reads $(\RG w^\ast)_e=\lmax\,w^\ast_e$ for
every edge $e$; and the tree formula \eqref{eq:tree-kappa}, when valid at $e$, is
precisely the identity $\kappa_e\,w^\ast_e=-(\RG w^\ast)_e$ (this is how \eqref{eq:RT} was
derived). Hence, whenever the tree formula holds at $e$,
\begin{equation}\label{eq:kappa-lam}
\kappa_e=-\frac{(\RG w^\ast)_e}{w^\ast_e}=-\lmax .
\end{equation}

($\Leftarrow$) If \eqref{eq:3sum} holds for every cycle edge in $w^\ast$, then by
Proposition~\ref{prop:balanced} the tree formula is valid on every edge --- pendant
edges always, cycle edges by \eqref{eq:3sum} --- so \eqref{eq:kappa-lam} gives
$\kappa_e=-\lmax$ for all $e$: the curvature is constant, with common value $-\lmax$.
For uniqueness within the balanced regime, let $w>0$ be any balanced Einstein metric of
curvature $\kappa$. The tree identity holds on every edge, so
$\RG w=-\kappa w$. Since $w$ is positive and $\RG+cI$ is nonnegative irreducible for
$c$ sufficiently large, Perron--Frobenius gives $-\kappa=\lmax$ and
$w\propto w^\ast$ (Lemma~\ref{lem:pf}).

($\Rightarrow$) Suppose some cycle edge $c_j$ violates \eqref{eq:3sum} in $w^\ast$.
Every pendant edge, and every cycle edge that does satisfy \eqref{eq:3sum}, still
obeys the tree formula, hence by \eqref{eq:kappa-lam} has curvature exactly $-\lmax$. The
violating edge $c_j$, by contrast, has by Proposition~\ref{prop:balanced}(ii) true
curvature strictly greater than its tree-formula value, which by
\eqref{eq:kappa-lam} equals $-\lmax$; thus $\kappa_{c_j}>-\lmax$. Therefore the curvature
is not identically $-\lmax$ --- it is $-\lmax$ on the pendant and balanced edges and
strictly larger on $c_j$ --- so $w^\ast$ is not the Einstein metric of curvature $-\lmax$.
\end{proof}

Two comments on the forward direction. First, why the curvature is non-constant
whenever \eqref{eq:3sum} fails and the graph has any pendant edge: that pendant edge
still has curvature exactly $-\lmax$ by \eqref{eq:kappa-lam}, while the violating cycle edge
has curvature strictly above $-\lmax$, two different values. The only graphs with no
pendant edge are the bare cycles; on a bare short cycle every edge violates \eqref{eq:3sum},
symmetry keeps the curvature constant, and $w^\ast$ is Einstein, but with a
constant different from $-\lmax$ (e.g.\ $K_3$ has $\kappa\equiv\tfrac32$ while $-\lmax=0$;
Section~\ref{sec:short}). Second, none of this concerns existence: when \eqref{eq:3sum}
fails, $\RG$ merely stops computing an Einstein metric; an Einstein metric may or may not
exist in a different transport regime (Section~\ref{sec:exist}).

\begin{corollary}\label{cor:sym}
For a regular decorated cycle --- every cycle vertex has the same degree $d$ and
carries the same rooted pendant tree --- the Perron vector is constant on the cycle, so
\eqref{eq:3sum} reads $3\le g/2$. Hence for girth $g\ge6$ the metric $w^\ast$ is a genuine
Einstein metric with $\kappa=-\lmax(\RG)$, for every $d$ and every such decoration.
\end{corollary}

\begin{proposition}\label{prop:apriori}
Let
\[
\rho_G=\frac{\max_j w^\ast_{c_j}}{\min_j w^\ast_{c_j}}
\]
be the spread of the Perron eigenvector over the cycle edges, a quantity determined by
the graph alone, since $\RG$ depends only on the degrees. If $g\ge 6\,\rho_G$, then
$w^\ast$ is balanced, hence a discrete Einstein metric (Theorem~\ref{thm:einstein}).
Corollary~\ref{cor:sym} is the case $\rho_G=1$.
\end{proposition}
\begin{proof}
For each $j$, $w^\ast_{c_{j-1}}+w^\ast_{c_j}+w^\ast_{c_{j+1}}\le 3\max_i w^\ast_{c_i}$, while
$L=\sum_i w^\ast_{c_i}\ge g\min_i w^\ast_{c_i}$. Hence every three-consecutive sum is at most
$L/2$ as soon as $3\max_i w^\ast_{c_i}\le\tfrac g2\min_i w^\ast_{c_i}$, i.e.\ $g\ge6\rho_G$;
Theorem~\ref{thm:einstein} then applies.
\end{proof}

\begin{remark}\label{rem:rho-unbounded}
The sufficient condition $g\ge6\rho_G$ is not tight, and $\rho_G$ is not the right quantity to
control. If every cycle vertex carries the same pendant tree, then $\rho_G=1$
(Corollary~\ref{cor:sym}). Consider instead the opposite extreme, which we call a
single-vertex defect: a long cycle that is bare except for one pendant tree attached
at a single vertex. There the Perron weights are largest on the two cycle edges at the
decorated vertex and decay geometrically --- by a factor $\mu>1$ per step, computed
in Proposition~\ref{prop:threshold} --- as one moves away from it along the cycle. The
smallest weight sits on the far side of the cycle, at distance about $g/2$, so
$\rho_G\asymp\mu^{g/2}$ grows exponentially in the girth, and $g\ge6\rho_G$ fails for all
large $g$, although such graphs are balanced whenever the decoration is mild. The
exponentially large spread is carried entirely by the exponentially small weight on the far
side of the cycle, which the three-edge window \eqref{eq:3sum} does not see: balance is a
local condition at the decorated vertex. Proposition~\ref{prop:threshold} makes this precise,
replacing $\rho_G$ by a local quantity with an exact threshold.
\end{remark}

\section{The periodic reduction}\label{sec:reduction}

Throughout this section we are in the balanced regime, so $\RG$ is the curvature
operator, and we compute its spectrum. Some terminology: by a pendant edge we mean
any edge not on the cycle, i.e.\ any edge of one of the pendant trees $B_v$, not
only the edges ending in leaves. Recall that $\RG$ acts on the edges of $G$, i.e.\
on the vertices of the line graph $L(G)$: the cycle edges $c_1,\dots,c_g$ form a $g$-cycle
in $L(G)$, and the pendant edges hang off it (Figure~\ref{fig:linegraph}). Order the
edge coordinates as $(\text{pendant edges}\mid\text{cycle edges }c_1,\dots,c_g)$. Since
distinct pendant trees couple only through the cycle, the pendant block of $\RG$ is
block-diagonal, one block per cycle vertex.

We Schur-eliminate the pendant-edge coordinates of $\RG z=\lambda z$, one pendant
tree at a time. At a root $v$ the branch equations see the two incident cycle coordinates
only through their sum $z_c+z_{c'}$; hence the eliminated tree contributes a single scalar
feedback $\Psi_v(\lambda)$, leaving a $g\times g$ eigenproblem on the cycle.

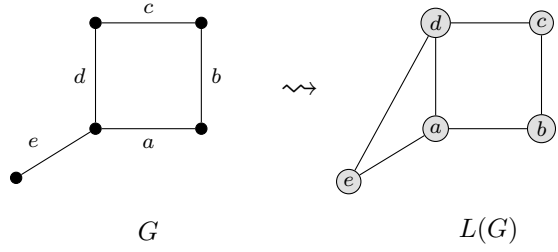
\begin{figure}[t]
\centering
\begin{tikzpicture}[scale=1.0,
  v/.style={circle,draw,fill=black,inner sep=1.5pt},
  e/.style={circle,draw,fill=black!12,inner sep=1.4pt,font=\scriptsize}]
 \begin{scope}
   \node[v](g0) at (0,0){}; \node[v](g1) at (1.4,0){}; \node[v](g2) at (1.4,1.4){};
   \node[v](g3) at (0,1.4){}; \node[v](g4) at (-1.05,-0.65){};
   \draw (g0)--node[below,font=\scriptsize]{$a$}(g1);
   \draw (g1)--node[right,font=\scriptsize]{$b$}(g2);
   \draw (g2)--node[above,font=\scriptsize]{$c$}(g3);
   \draw (g3)--node[left,font=\scriptsize]{$d$}(g0);
   \draw (g0)--node[above left,font=\scriptsize,pos=0.6]{$e$}(g4);
   \node[font=\small] at (0.7,-1.35) {$G$};
 \end{scope}
 \node at (2.7,0.5) {\Large$\rightsquigarrow$};
 \begin{scope}[shift={(4.5,0)}]
   \node[e](a) at (0,0){$a$}; \node[e](b) at (1.4,0){$b$};
   \node[e](c) at (1.4,1.4){$c$}; \node[e](d) at (0,1.4){$d$};
   \node[e](ee) at (-1.15,-0.7){$e$};
   \draw (a)--(b)--(c)--(d)--(a);
   \draw (ee)--(a); \draw (ee)--(d);
   \node[font=\small] at (0.7,-1.35) {$L(G)$};
 \end{scope}
\end{tikzpicture}
\caption{A unicyclic graph $G$ (here $C_4$ with one pendant edge $e$) and its line graph
$L(G)$. The cycle edges $a,b,c,d$ become a $g$-cycle in $L(G)$; each pendant edge (here
$e$) attaches to the two line-graph vertices at its foot. The Ricci matrix $\RG$ is a
(Schr\"odinger-type) operator on $L(G)$.}
\label{fig:linegraph}
\end{figure}

\begin{figure}[t]
\centering
\begin{tikzpicture}[scale=1.15, cv/.style={circle,draw,fill=black,inner sep=1.6pt}]
 \foreach \k in {1,...,6} {\coordinate (p\k) at ({90+60*\k}:1.5);}
 \foreach \k in {1,...,6} {
   \pgfmathsetmacro{\ang}{90+60*\k}
   \fill[gray!18,draw=gray!55] (p\k) -- (\ang-11:2.5) -- (\ang+11:2.5) -- cycle;
   \node[font=\scriptsize] at (\ang:2.03) {$\Psi_{\k}$};
 }
 \draw[very thick] (p1)--(p2)--(p3)--(p4)--(p5)--(p6)--(p1);
 \foreach \k in {1,...,6} {\node[cv] at (p\k) {};}
 \node[font=\scriptsize] at ({90+60*1.5}:1.16) {$z_1$};
 \node[font=\scriptsize] at ({90+60*2.5}:1.16) {$z_2$};
\end{tikzpicture}
\caption{The reduction of Section~\ref{sec:reduction}: Schur-eliminating each pendant tree
replaces it by a scalar feedback $\Psi_j(\lambda)$ (Lemma~\ref{lem:feedback}, the shaded
branches), leaving the periodic three-term recurrence \eqref{eq:recur} on the cycle-edge
coordinates $z_1,\dots,z_g$. Away from the branch spectra and the poles of the transfer
matrices, the eigenvalues of $\RG$ are the
$\lambda$ with $\operatorname{tr}\!\big[T_g(\lambda)\cdots T_1(\lambda)\big]=2$
(Theorem~\ref{thm:reduction}).}
\label{fig:reduction}
\end{figure}
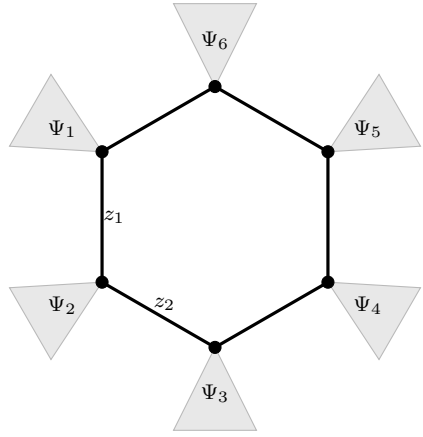

\subsection{Branch feedback}
We eliminate the pendant trees one at a time, leaving on the cycle only a scalar feedback per
vertex. Fix an eigenvector $z$ of $\RG$ at eigenvalue $\lambda$, so $z$ assigns a coordinate
$z_e$ to each edge. Here $B_v$ denotes the rooted pendant-tree subgraph attached to the cycle vertex $v$; it
meets the rest of $G$ only at $v$. We write
\[
 D_v:=R_G[E(B_v),E(B_v)]
\]
for the principal branch block of $R_G$; all degrees in $D_v$ are the degrees in the original
graph $G$, not the degrees of $B_v$ viewed as an isolated tree. In $\RG$ any two edges
sharing $v$ are coupled by the same entry $1/d_v$; hence
the eigen-equations inside $B_v$ see the cycle only through the combination $z_c+z_{c'}$, where
$c,c'$ are the two cycle edges at $v$. Concretely: collecting the rows of $\RG z=\lambda z$
indexed by the edges of $B_v$ gives a square linear system for the branch coordinates in which
the cycle coordinates enter only through the single scalar $z_c+z_{c'}$, multiplied by fixed
coefficients. For $\lambda$ outside the (finite) spectrum of $D_v$ this system is
uniquely solvable, and since its right-hand side is linear in the scalar $z_c+z_{c'}$, so
is its solution --- hence so is any linear functional of the solution. Writing
$\sigma_v=\sum_r z_r$ for the sum of the coordinates of the pendant edges incident to $v$
(Figure~\ref{fig:feedback-proof}), this gives
\[
\sigma_v=\Psi_v(\lambda)\,(z_c+z_{c'})
\]
for a scalar $\Psi_v(\lambda)$ depending only on $\lambda$ and on the shape of $B_v$, the
Schur response, or feedback, of the branch. The next lemma computes it as a
continued fraction by performing the elimination leaf-by-leaf. For a vertex carrying
$p=d_v-2$ pendant leaves, where $p$ is the number of such leaves, the elimination is
a single step and gives
$\Psi_v(\lambda)=p/(d_v\lambda+d_v+2-p)$, as derived at the end of
Lemma~\ref{lem:feedback}.

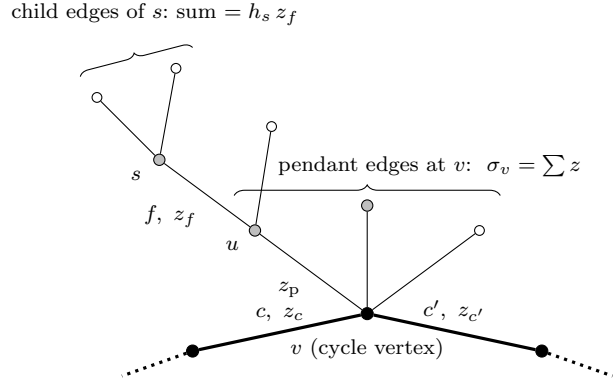
\begin{figure}[t]
\centering
\begin{tikzpicture}[scale=1.1,
  cv/.style={circle,draw,fill=black,inner sep=1.6pt},
  iv/.style={circle,draw,fill=black!25,inner sep=1.5pt},
  lf/.style={circle,draw,fill=white,inner sep=1.3pt},
  el/.style={font=\scriptsize}]
 \node[cv] (v) at (0,0) {};
 \node[el] at (0,-0.42) {$v$\ (cycle vertex)};
 \node[cv] (cl) at (-2.1,-0.45) {};
 \node[cv] (cr) at (2.1,-0.45) {};
 \draw[very thick] (v)--(cl) node[el,midway,above]{$c,\ z_c$};
 \draw[very thick] (v)--(cr) node[el,midway,above]{$c',\ z_{c'}$};
 \draw[very thick,dotted] (cl)--(-2.95,-0.75); \draw[very thick,dotted] (cr)--(2.95,-0.75);
 \node[iv] (u) at (-1.35,1.0) {};
 \node[el] at (-1.62,0.82) {$u$};
 \node[iv] (u2) at (0,1.3) {};
 \node[lf] (u3) at (1.35,1.0) {};
 \draw (u2)--(v); \draw (u3)--(v);
 \draw (v)--(u) node[el,midway,below left]{$z_{\mathrm p}$};
 \draw[decorate,decoration={brace,amplitude=4pt}] (-1.6,1.42) -- (1.6,1.42);
 \node[el] at (0.75,1.78) {pendant edges at $v$:\ \ $\sigma_v=\textstyle\sum z$};
 \node[iv] (s1) at (-2.5,1.85) {};
 \node[el] at (-2.78,1.66) {$s$};
 \node[lf] (s2) at (-1.15,2.25) {};
 \draw (u)--(s1) node[el,midway,below left]{$f,\ z_f$};
 \draw (u)--(s2);
 \node[lf] (d1) at (-3.25,2.6) {};
 \node[lf] (d2) at (-2.3,2.95) {};
 \draw (s1)--(d1); \draw (s1)--(d2);
 \draw[decorate,decoration={brace,amplitude=3.5pt}] (-3.45,2.85) -- (-2.1,3.2);
 \node[el] at (-2.55,3.6) {child edges of $s$:\ sum $=h_s\,z_f$};
\end{tikzpicture}
\caption{The objects of Lemma~\ref{lem:feedback}, on one pendant tree $B_v$ rooted at the
cycle vertex $v$ (edges oriented away from $v$). Each vertex $u\neq v$ has one parent
edge, of coordinate $z_{\mathrm p}$, and child edges; $f=\{u,s\}$ is a child edge of
$u$, and the induction ratio $h_s$ expresses the sum of the child-edge coordinates of $s$ as
$h_s z_f$. At the root, $\sigma_v$ is the sum over the pendant edges at $v$, and the two
cycle edges $c,c'$ play the role of the parent.}
\label{fig:feedback-proof}
\end{figure}

\begin{lemma}\label{lem:feedback}
Orient the pendant tree $B_v$ away from its root $v$, so that every vertex $u\neq v$ has one
parent edge and (possibly zero) child edges (Figure~\ref{fig:feedback-proof}). Define
recursively: for a leaf $u$, $h_u=0$; for an internal vertex $u$ with children $s$,
\[
h_u=\frac{B_u}{1-B_u},\qquad
B_u=\sum_{s\text{ child of }u}\frac{1}{d_u\lambda+2+d_u(1-h_s)/d_s}.
\]
Then, with $u$ running over the pendant children of $v$,
\[
\Psi_v(\lambda)=\frac{A_v}{1-A_v},\qquad
A_v=\sum_{u}\frac{1}{d_v\lambda+2+d_v(1-h_u)/d_u}.
\]
In particular, if $v$ carries $p=d_v-2$ pendant leaves, then
$\Psi_v(\lambda)=\dfrac{p}{d_v\lambda+d_v+2-p}=\dfrac{d_v-2}{d_v\lambda+4}$.
\end{lemma}
\begin{proof}
We prove, by induction from the leaves towards the root, the following claim: \emph{for every
vertex $u\neq v$ of $B_v$, with parent-edge coordinate $z_{\mathrm p}$,}
\[
\sum_{f\ \text{child edge of}\ u} z_f\ =\ h_u\,z_{\mathrm p},
\]
with $h_u$ as defined in the statement (all denominators are nonzero for $\lambda$
outside a finite set; the resulting equality then holds on the common domain and hence as an
identity of rational functions of $\lambda$).

A leaf $u$ has no child edges, so the sum is empty and $h_u=0$. For an internal vertex $u$ with
children $s$, joined by the child edges $f=\{u,s\}$, suppose the claim holds at every child
(Figure~\ref{fig:feedback-proof}) and write $T_u=z_{\mathrm p}+\sum_f z_f$ for the sum of the
coordinates of all edges at $u$. The row of $\RG z=\lambda z$ indexed by $f=\{u,s\}$ has three
kinds of terms: the diagonal entry $-(1/d_u+1/d_s)$; one entry $1/d_u$ for each other edge sharing
$u$, contributing $\tfrac1{d_u}(T_u-z_f)$ in total; and one entry $1/d_s$ for each other edge
sharing $s$, namely the child edges of $s$. By the inductive hypothesis their coordinates
sum to $h_s z_f$, where $h_s$ is the child-to-parent ratio defined recursively in the
statement. Hence
\[
\lambda z_f=-\Bigl(\tfrac1{d_u}+\tfrac1{d_s}\Bigr)z_f+\tfrac1{d_u}\bigl(T_u-z_f\bigr)
+\tfrac1{d_s}\,h_s z_f,
\qquad\text{i.e.}\qquad
z_f\Bigl(d_u\lambda+2+\tfrac{d_u(1-h_s)}{d_s}\Bigr)=T_u .
\]
Each child edge carries a scalar multiple of $T_u$, with the scalar depending on the child
$s$, so summing these multiples over the children gives
$\sum_f z_f=B_u\,T_u$; substituting $T_u=z_{\mathrm p}+\sum_f z_f$ and solving,
$\sum_f z_f=z_{\mathrm p}\,B_u/(1-B_u)=h_u z_{\mathrm p}$, the claim at $u$.

At the root the same computation applies, with the single parent edge replaced by the pair of
cycle edges: $T_v=z_c+z_{c'}+\sigma_v$, and each pendant edge $\{v,u\}$ satisfies
$z_{\{v,u\}}\bigl(d_v\lambda+2+d_v(1-h_u)/d_u\bigr)=T_v$, whence
$\sigma_v=A_v(z_c+z_{c'}+\sigma_v)$ and $\sigma_v=\tfrac{A_v}{1-A_v}(z_c+z_{c'})$, that is
$\Psi_v=A_v/(1-A_v)$. When $v$ carries $p$ pendant leaves, $h_u=0$ and $d_u=1$ make each
denominator equal to $d_v\lambda+d_v+2$, so $A_v=p/(d_v\lambda+d_v+2)$ and
$\Psi_v=p/(d_v\lambda+d_v+2-p)$, which for $p=d_v-2$ is $(d_v-2)/(d_v\lambda+4)$.
\end{proof}

\subsection{Reduction to a periodic Jacobi matrix}
Write $a_j=1/d_{v_j}$ and $\Psi_j=\Psi_{v_j}$. Substituting
$\sigma_{v}=\Psi_v(z+z')$ into the eigen-equation of the cycle edge
$c_j=\{v_{j-1},v_j\}$ (coordinate $z_j$) collects the pendant contributions into local
terms, giving the three-term recurrence
\begin{equation}\label{eq:recur}
P_j\,z_{j+1}=Q_j\,z_j-P_{j-1}\,z_{j-1},\qquad
P_j=\frac{1+\Psi_j}{d_{v_j}},\quad
Q_j=\lambda+a_{j-1}+a_j-a_{j-1}\Psi_{j-1}-a_j\Psi_j,
\end{equation}
i.e.\ $\RG$ restricted to the cycle is the $g\times g$ cyclically tridiagonal matrix
$M(\lambda)$ with $M_{jj}=Q_j$, $M_{j,j\pm1}$ the appropriate $-P$. Set the transfer
matrices $T_j(\lambda)=\begin{psmallmatrix}Q_j/P_j & -P_{j-1}/P_j\\ 1&0\end{psmallmatrix}$.

\begin{theorem}\label{thm:reduction}
Let $\lambda\notin\bigcup_v\sigma(D_v)$ and assume $P_j(\lambda)\ne0$ for every
$j$, so that $\Pi(\lambda)=T_g(\lambda)\cdots T_1(\lambda)$ is defined. Then
$\det\Pi(\lambda)=1$, and $\lambda$ is an eigenvalue of $\RG$ if and only if
\[
\operatorname{tr}\Pi(\lambda)=2.
\]
For such an eigenvalue the restriction of its eigenvector to the cycle is nonzero, and
the pendant coordinates are reconstructed from it by Schur elimination. In particular
$\lmax(\RG)$ avoids the branch spectra, has $P_j(\lmax)>0$, and is the largest such
root; it is the zero-Bloch-phase (periodic Perron) mode.
\end{theorem}
\begin{proof}
Schur-eliminating the (block-diagonal) pendant block of $\lambda I-\RG$ yields
$\det(\lambda I-\RG)=\big(\prod_v\det(\lambda I-D_v)\big)\cdot\det M(\lambda)$, and by
Lemma~\ref{lem:feedback} the Schur complement is exactly the cyclically tridiagonal
$M(\lambda)$ of \eqref{eq:recur}. For $\lambda\notin\bigcup_v\sigma(D_v)$ the prefactor
is nonzero, so $\lambda\in\sigma(\RG)$ iff $\det M(\lambda)=0$. Moreover, an eigenvector
with zero cycle restriction would have zero pendant coordinates because the pendant
block is invertible, and hence would vanish identically. Thus its cycle restriction is
nonzero, and the eliminated pendant coordinates are uniquely reconstructed from it. The
recurrence \eqref{eq:recur} is $M(\lambda)z=0$; its solutions are generated by the
transfer matrices, and $\det T_j=P_{j-1}/P_j$ telescopes to $\det\Pi=1$. A nonzero
periodic solution $z_{j+g}=z_j$ exists iff $\Pi(\lambda)$ fixes $(z_1,z_0)^{\!\top}$,
i.e.\ has eigenvalue $1$; since $\det\Pi=1$ its eigenvalues are $\rho,\rho^{-1}$ and
$\rho=1$ iff $\operatorname{tr}\Pi=2$.

Finally choose $c$ so that $\RG+cI$ is nonnegative and irreducible. Every branch block
$D_v+cI$ is a proper principal submatrix, so strict Perron comparison gives
$\lmax(D_v)<\lmax(\RG)$. Hence $\lmax(\RG)$ avoids every branch spectrum. At the
positive Perron eigenvector, each eliminated branch sum is nonnegative, so
$P_j(\lmax)=(1+\Psi_j(\lmax))/d_{v_j}>0$. The Perron vector is periodic with zero
Bloch phase, and its eigenvalue is the largest root.
\end{proof}

\section{Closed forms and asymptotics}\label{sec:closed}

\begin{proposition}\label{prop:bare}
For every positive weighting of $C_g$, the fixed local matrix $R_{C_g}$ is the same
circulant matrix, with diagonal $-1$ and $\tfrac12$ on the two cyclic neighbours. Its
spectrum is
\[
 \bigl\{\cos(2\pi k/g)-1:k=0,\ldots,g-1\bigr\},
\]
and its simple spectral top is $0$, with constant eigenvector. For $g\ge6$ the uniform
Perron weighting is balanced and therefore is an Einstein metric of curvature $0$.
\end{proposition}
\begin{proof}
Every vertex of $C_g$ has degree two, so the definition of $R_G$ gives the stated circulant
matrix. Its Fourier vectors $(e^{2\pi i k j/g})_{j=0}^{g-1}$ have eigenvalues
$-1+\cos(2\pi k/g)$; hence the spectral top is the simple value $0$ at $k=0$. For the
uniform weighting the three-edge condition is $3\le g/2$, so $g\ge6$ is balanced and
Theorem~\ref{thm:einstein} gives curvature $0$.
\end{proof}

\noindent This is the base case of the periodic reduction, and reproduces the curvature of the
(unweighted) cycle computed by Lin--Lu--Yau \cite{LLY}. Here ``Bloch spectrum'' refers
to the decomposition into periodic phase modes around the cycle; only this passage on
the line graph is new.

For the homogeneous decorated cycle (degree $d$, identical feedback $\Psi$) the
recurrence \eqref{eq:recur} is constant-coefficient, $z_{j-1}+z_{j+1}=\mu(\lambda)z_j$
with
\[
\mu(\lambda)=\frac{d\lambda+2-2\Psi(\lambda)}{1+\Psi(\lambda)},
\]
and the periodic condition is $\mu(\lambda)=2\cos(2\pi k/g)$, $k=0,\dots,g-1$; the Perron
value is the $k=0$ root of $\mu(\lambda)=2$.

\begin{figure}[t]
\centering
\begin{tikzpicture}[scale=1.0,
  cv/.style={circle,draw,fill=black,inner sep=1.5pt},
  lf/.style={circle,draw,fill=white,inner sep=1.2pt}]
 \foreach \k in {0,...,5} {\node[cv] (n\k) at ({90+60*\k}:1.4) {};}
 \draw[very thick] (n0)--(n1)--(n2)--(n3)--(n4)--(n5)--(n0);
 \foreach \k in {0,...,5} {
   \node[lf] (a\k) at ({90+60*\k-12}:2.15) {};
   \node[lf] (b\k) at ({90+60*\k+12}:2.15) {};
   \draw (n\k)--(a\k); \draw (n\k)--(b\k);
 }
\end{tikzpicture}
\caption{The regular sun $S_{6,4}$: the cycle $C_6$ with $d-2=2$ pendant leaves at each
vertex (so each cycle vertex has degree $d=4$). By Theorem~\ref{thm:sun},
$\lmax=\tfrac{2(\sqrt{d-1}-1)}{d}$, independent of the girth $g$.}
\label{fig:sun}
\end{figure}
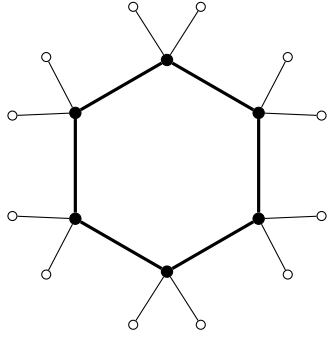

\begin{theorem}\label{thm:sun}
Let $S_{g,d}$ be the cycle $C_g$ with $d-2$ pendant leaves at each vertex (so every cycle
vertex has degree $d\ge3$; Figure~\ref{fig:sun}). Then, for every $g\ge3$,
\[
\lmax(R_{S_{g,d}})=\frac{2\big(\sqrt{d-1}-1\big)}{d},
\]
independent of $g$, and the corresponding eigenvector is positive. For $g\ge6$ this Perron
vector is balanced, hence it is a discrete Einstein metric with
$\kappa(S_{g,d})=-\lmax<0$. For $3\le g\le5$, the same spectral formula concerns the fixed
local matrix $R_{S_{g,d}}$, but $-\lmax(R_{S_{g,d}})$ need not equal the curvature of the
unbalanced Einstein metric.
\end{theorem}
\begin{proof}
The orbit reduction for the fixed matrix $R_{S_{g,d}}$ is valid for every $g\ge3$.
Here $\Psi(\lambda)=(d-2)/(d\lambda+4)$ by Lemma~\ref{lem:feedback}. The Perron condition
$\mu(\lambda)=2$ is $d\lambda+2-2\Psi=2+2\Psi$, i.e.\ $d\lambda=4\Psi(\lambda)$, giving
$d\lambda(d\lambda+4)=4(d-2)$, i.e.\ $d^2\lambda^2+4d\lambda-4(d-2)=0$. The positive root
is $\lambda=2(\sqrt{d-1}-1)/d$. It is the $k=0$ Bloch mode, whose eigenvector is constant
on the cycle and positive on the leaves, so it is the Perron value. For $g\ge6$,
Corollary~\ref{cor:sym} places this vector in the balanced regime, and therefore
$\kappa=-\lmax$.
\end{proof}

\begin{corollary}\label{cor:flat}
The function $F(d)=2(\sqrt{d-1}-1)/d$ satisfies $F(d)\to0^+$ as $d\to\infty$, and over
real $d$ attains its maximum at $d=4+2\sqrt2\approx6.83$; over integers the maximum is
$F(7)=\tfrac{2(\sqrt6-1)}{7}$. Thus, for every fixed $g\ge6$, $\{S_{g,d}\}_{d}$ is a family of Einstein
unicyclic graphs with $\kappa\to0^-$. The spectral bound
$\lmax(R_{S_{g,d}})\le\tfrac{2(\sqrt6-1)}{7}$ holds for all $g\ge3$ and $d\ge3$.
\end{corollary}
\begin{proof}
Write $t=\sqrt{d-1}$, so $F=2(t-1)/(t^2+1)$; then $F'=0$ at $t^2-2t-1=0$, i.e.\
$t=1+\sqrt2$, $d=4+2\sqrt2$, and $F\to0$ as $t\to\infty$. Comparing integer neighbours,
$F(7)>F(6),F(8)$, giving the stated integer maximum.
\end{proof}

The value $\tfrac{2(\sqrt6-1)}{7}$ bounds only the leaf-decorated (sun) family; deeper
branches exceed it. The symmetric $2$-ary decorated $C_6$ --- each cycle vertex carrying a
binary tree of depth $2$, balanced by Corollary~\ref{cor:sym} --- is a genuine Einstein
metric whose $\lmax(\RG)$ exceeds this value. The correct universal bound is the same degree bound
as for trees, which we now prove.

\begin{theorem}\label{thm:degbound}
For any unicyclic graph $G$ with maximum degree $\mathcal D=\max_v d_v$,
\[
\lmax(\RG)\ \le\
\begin{cases}
1-\dfrac{4}{\mathcal D}+\dfrac{2\sqrt{\mathcal D-1}}{\mathcal D}, & 2\le\mathcal D\le18,\\[8pt]
\dfrac{15+6\sqrt2}{19}, & \mathcal D\ge19,
\end{cases}
\]
the same bound as for trees \cite{BCH}.
\end{theorem}
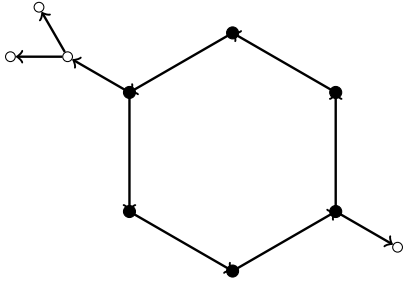
\begin{figure}[H]
\centering
\begin{tikzpicture}[scale=1.05,
  cv/.style={circle,draw,fill=black,inner sep=1.6pt},
  lf/.style={circle,draw,fill=white,inner sep=1.3pt},
  a/.style={->,line width=0.9pt}]
 \foreach \k in {1,...,6} {\coordinate (p\k) at ({90+60*\k}:1.5);}
 \foreach \k in {1,...,6} {\pgfmathtruncatemacro{\n}{mod(\k,6)+1}\draw[a] (p\k)--(p\n);}
 \foreach \k in {1,...,6} {\node[cv] at (p\k) {};}
 \node[lf] (q1) at ($(p1)+(150:0.9)$) {}; \draw[a] (p1)--(q1);
 \node[lf] (q2) at ($(q1)+(120:0.72)$) {}; \draw[a] (q1)--(q2);
 \node[lf] (q3) at ($(q1)+(180:0.72)$) {}; \draw[a] (q1)--(q3);
 \node[lf] (r1) at ($(p4)+(330:0.9)$) {}; \draw[a] (p4)--(r1);
\end{tikzpicture}
\caption{The orientation used in Theorem~\ref{thm:degbound}: the cycle is directed (each
cycle vertex receives one in-edge from the cycle), and every pendant tree is oriented away
from the cycle. Because $|E|=|V|$, every vertex has exactly one in-edge (its
``parent'') --- the cycle absorbs the root, so there is no root-edge exception.}
\label{fig:orient}
\end{figure}
\begin{proof}
Collecting the diagonal and off-diagonal contributions of $\RG$ at each vertex gives, for
any $f\in\R^{E(G)}$, the vertex decomposition, valid for every graph,
\begin{equation}\label{eq:vdecomp}
\langle f,\RG f\rangle=\sum_{v\in V}\frac{S_v^2-2A_v}{d_v},
\qquad S_v=\sum_{e\ni v}f_e,\quad A_v=\sum_{e\ni v}f_e^2 .
\end{equation}
Since $|E(G)|=|V(G)|$, we may orient the edges so that every vertex has exactly one
in-edge: direct the unique cycle as a directed cycle, and orient each pendant tree away
from the cycle (Figure~\ref{fig:orient}). Then every vertex $v$ has a unique parent
edge $e_p(v)$ (its in-edge), and its remaining incident edges are children
$E_c(v)$; dually, each edge is the parent of exactly one vertex and a child of exactly one.
On a tree ($|E|=|V|-1$), any orientation away from a root leaves that root without an
in-edge, so the edges adjacent to it require a separate estimate \cite{BCH}. A unicyclic
graph has one additional edge: orienting the cycle cyclically gives every cycle vertex,
including the roots of all pendant trees, exactly one in-edge. Thus the same estimate
applies at every vertex, with no root exception.

For any $\alpha_v>0$, the weighted Cauchy--Schwarz inequality
\[
S_v^2\le(1+\alpha_v)f_{e_p(v)}^2
 +(1+\tfrac1{\alpha_v})(d_v-1)\sum_{e\in E_c(v)}f_e^2
\]
turns \eqref{eq:vdecomp} into
\[
\begin{aligned}
\frac{S_v^2-2A_v}{d_v}
&\le B(d_v,\alpha_v)\,f_{e_p(v)}^2
  +A(d_v,\alpha_v)\sum_{e\in E_c(v)}f_e^2,\\
B(d,\alpha)&=\frac{\alpha-1}{d},
&\qquad A(d,\alpha)&=\frac{(1+\tfrac1\alpha)(d-1)-2}{d}.
\end{aligned}
\]
Summing over $v$ and regrouping by edge --- each $e=(u\!\to\!v)$ is a child of its tail $u$
and the parent of its head $v$ ---
\[
\langle f,\RG f\rangle\ \le\ \sum_{e=(u\to v)}\bigl(A(d_u,\alpha_u)+B(d_v,\alpha_v)\bigr)f_e^2 .
\]
With the linear choice $\alpha_d=Kd+1$ ($K\ge0$) we have $B(d,\alpha_d)=K$, so the coefficient
of $f_e^2$ equals $A(d_u,Kd_u+1)+K=C(d_u,K)$, where
$C(d,K)=\big((1+\tfrac1{Kd+1})(d-1)-2\big)/d+K$ depends only on the tail degree. Hence
\[
\lmax(\RG)=\max_{\|f\|=1}\langle f,\RG f\rangle\ \le\ \max_{v\in V}C(d_v,K).
\]
By the optimization of $C(d,K)$ in Lemma~3.2 of \cite{BCH}, taking
$K=K_m:=(\sqrt{m-1}-1)/m$ with
$m=\min(\mathcal D,19)$ gives $\max_{1\le d\le\mathcal D}C(d,K_m)=C(m,K_m)=1-\tfrac4m+\tfrac{2\sqrt{m-1}}{m}$,
which is the asserted bound.

For each fixed $d\ge2$ the best constant this method can produce, allowing $K=0$ when
$d=2$, is
$B(d):=\min_{K\ge0} C(d,K)=1-\tfrac4d+\tfrac{2\sqrt{d-1}}{d}$, and the function $B$ increases up to
its real maximum at $d=10+4\sqrt5\approx18.94$ (substituting $t=\sqrt{d-1}$, $B'=0$ reads
$t^2-4t-1=0$) and decreases towards $1$ beyond it; over the integers its peak is $B(19)$.
Thus for $\mathcal D\le19$ the largest value of $B(d)$ among the allowed integer degrees
occurs at $d=\mathcal D$. Once $\mathcal D\ge19$, the function has already reached its
integer maximum at $d=19$, and allowing still larger degrees does not increase the
bound. This is why the formula uses $m=\min(\mathcal D,19)$.
\end{proof}

\begin{proposition}\label{prop:sharp}
The degree bound in Theorem~\ref{thm:degbound} is sharp for every prescribed maximum
degree $\mathcal D\ge2$. Writing $m=\min\{\mathcal D,19\}$, there are unicyclic graphs
$G_L$ with maximum degree exactly $\mathcal D$ such that
\[
 \lmax(R_{G_L})\xrightarrow[L\to\infty]{}
 B(m):=1-\frac4m+\frac{2\sqrt{m-1}}m .
\]
\end{proposition}
\begin{proof}
For $\mathcal D=2$, take $G_L=C_{L+3}$, for which $B(2)=0$ and
$\lmax(R_{G_L})=0$. Assume $\mathcal D\ge3$ and put $m=\min\{\mathcal D,19\}$.
Start with a fixed cycle, say $C_6$, and identify one cycle vertex with the root of a
rooted tree of depth $L$. Give that root $m-2$ children and every other non-leaf vertex
$m-1$ children.
Thus every vertex in the bulk of this branch has degree $m$. If $\mathcal D>19$, attach
$\mathcal D-2$ additional leaves at a different cycle vertex; this finite decoration makes
the maximum degree exactly $\mathcal D$ without changing the growing $19$-regular bulk.
For $\mathcal D\le19$ no second decoration is needed.

Theorem~\ref{thm:degbound} gives $\lmax(R_{G_L})\le B(m)$. For the reverse
inequality we give an explicit family of test vectors. On the infinite $m$-regular tree
$\mathbb T_m$ one has
\[
 R_{\mathbb T_m}=\frac1m\bigl(A_{L(\mathbb T_m)}-2I\bigr),
\]
where $A_{L(\mathbb T_m)}$ is the adjacency operator of the line graph. Fix an edge
$e_0$, put $r(e)=d_{L(\mathbb T_m)}(e,e_0)$, and, for $q\ge1$, define
\[
 f_q(e)=
 \begin{cases}
 (m-1)^{-r(e)/2},&0\le r(e)\le q,\\
 0,&r(e)>q.
 \end{cases}
\]
For $r\ge1$ the sphere $\{e:r(e)=r\}$ has $2(m-1)^r$ edges. At every edge with
$1\le r(e)\le q-1$, one line-graph neighbour lies in the preceding sphere,
$m-2$ lie in the same sphere, and $m-1$ lie in the next sphere. Consequently
\[
 A_{L(\mathbb T_m)}f_q
 =\bigl(m-2+2\sqrt{m-1}\bigr)f_q
\]
away from the root and the outer boundary. Each radial sphere contributes a constant amount
to $\|f_q\|_2^2$, so $\|f_q\|_2^2\asymp q$, whereas the total root and boundary
error is $O(1)$. Therefore
\[
 \frac{\langle f_q,R_{\mathbb T_m}f_q\rangle}{\|f_q\|_2^2}
 =1-\frac4m+\frac{2\sqrt{m-1}}m+O(q^{-1})
 =B(m)+O(q^{-1}).
\]
For fixed $q$ and $L$ sufficiently large, a radius-$q$ edge neighbourhood in the bulk of
the growing branch of $G_L$ is isomorphic, with the same vertex degrees, to the corresponding
neighbourhood in $\mathbb T_m$. Transplanting $f_q$ there and extending it by zero therefore
preserves its Rayleigh quotient. First let $L\to\infty$ and then $q\to\infty$ to obtain
\[
 \liminf_{L\to\infty}\lmax(R_{G_L})\ge B(m).
\]
Together with the universal upper bound this proves the limit. The bulk support avoids the
degree mismatch at both the cycle root and the truncated boundary.
\end{proof}

The degree bound is matched by a companion lower bound: $\lmax(\RG)$ is always
nonnegative, and only the bare cycle sits at the bottom.

\begin{proposition}\label{prop:lower}
For every unicyclic graph $G$,
\[
\lmax(\RG)\ \ge\ 0,
\]
with equality if and only if $G$ is a bare cycle $C_g$ (of any girth $g\ge3$). Together with
Theorem~\ref{thm:degbound} this brackets the spectrum,
\[
0\ \le\ \lmax(\RG)\ \le\ 1-\frac4m+\frac{2\sqrt{m-1}}{m},\qquad m=\min(\mathcal D,19),
\]
the lower end attained only by the bare cycle. Consequently, in the balanced regime
(Theorem~\ref{thm:einstein}) the Einstein curvature $\kappa=-\lmax(\RG)$ satisfies
$\kappa\le0$, with $\kappa=0$ if and only if $G=C_g$ with $g\ge6$: every non-bare balanced
unicyclic graph has strictly negative Einstein curvature, and the regular suns of
Corollary~\ref{cor:flat} approach flatness from below.
\end{proposition}
\begin{proof}
Apply the vertex decomposition \eqref{eq:vdecomp} to the all-ones vector
$f\equiv\mathbf1$: then $S_v=A_v=d_v$, each summand is $(d_v^2-2d_v)/d_v=d_v-2$, and
\[
\langle\mathbf1,\RG\mathbf1\rangle=\sum_{v\in V}(d_v-2)=2|E|-2|V|=0,
\]
since $|E|=|V|$. Hence $\lmax(\RG)=\max_{\|f\|=1}\langle f,\RG f\rangle\ge
\langle\mathbf1,\RG\mathbf1\rangle/\|\mathbf1\|^2=0$. If $\lmax(\RG)=0$, the maximum of the
Rayleigh quotient is attained at $\mathbf1$, so $\mathbf1$ is a top eigenvector and
$\RG\mathbf1=0$. The entry of $\RG\mathbf1$ at an edge $e=\{x,y\}$ is
\[
-\Big(\tfrac1{\dx}+\tfrac1{\dy}\Big)
+\Big(\tfrac{\dx-1}{\dx}+\tfrac{\dy-1}{\dy}\Big)=2-\tfrac2{\dx}-\tfrac2{\dy},
\]
so $\RG\mathbf1=0$ forces $\tfrac1{\dx}+\tfrac1{\dy}=1$ at every edge. A pendant vertex
($\dx=1$) would force $\tfrac1{\dy}=0$, which is impossible; hence $\delta(G)\ge2$, and
since $\sum_{v}d_v=2|E|=2|V|$ every vertex has degree exactly $2$, so $G$ is $2$-regular
and connected, a bare cycle. Conversely $R_{C_g}\mathbf1=0$ by
Proposition~\ref{prop:bare}. The curvature statement is Theorem~\ref{thm:einstein}
($\kappa=-\lmax$ in the balanced regime) with the balanced-cycle threshold $g\ge6$ of
Proposition~\ref{prop:bare}.
\end{proof}

The curvature conclusion above is deliberately restricted to the balanced regime. Outside it,
$\lmax(R_G)$ no longer controls $\kappa$; the relevant object is the active regime operator
$R_G^{(\sigma)}$. Thus the bare short cycles $C_3,C_4,C_5$ are basic positively curved
examples, but they are not the only ones. For instance, Proposition~\ref{prop:sununiq} gives
for the symmetric Einstein metric on $S_{3,3}$
\[
   a^\ast=\frac{1+\sqrt{17}}4,
   \qquad \kappa=\frac{7-\sqrt{17}}6>0 .
\]

Proposition~\ref{prop:bare} and Theorem~\ref{thm:sun} are the periodic analogues of the
tree results $\lmax(P_n)=-1+\cos(\pi/(n-1))$ and the $d$-regular tree limit
$1-4/d+2\sqrt{d-1}/d$ of \cite{BCH}: the cycle replaces the Dirichlet chain by a periodic
(Bloch) chain. Two different ``peak degrees'' appear above and should not be confused. The
sun value $2(\sqrt{d-1}-1)/d$ --- the best a one-layer decoration of leaves can do
--- peaks at $d=7$ (Corollary~\ref{cor:flat}). The universal bound
$B(d)=1-4/d+2\sqrt{d-1}/d$ of Theorem~\ref{thm:degbound} --- approached only by deep
$d$-regular branches (Proposition~\ref{prop:sharp}) --- peaks at $d=19$, which is why the
cap $m=\min(\mathcal D,19)$ appears in its statement. Shallow decorations peak at $7$, the
supremum over all decorations at $19$; both numbers are inherited from the tree theory.

We close the section with a sharp balance criterion for a single-vertex defect (a long bare
cycle with one pendant tree at one vertex, as in Remark~\ref{rem:rho-unbounded}), pinning the
exact threshold left implicit in Proposition~\ref{prop:apriori}. Away from the decoration every
cycle edge obeys the same equation. As $g\to\infty$, the spectra of the homogeneous cycle
operators become dense in $[-2,0]$ (Proposition~\ref{prop:bare}); this is the essential
spectrum of the homogeneous line. Attaching a finite pendant tree is a finite-rank
perturbation, so the essential spectrum remains $[-2,0]$, while finitely many isolated
eigenvalues may appear above $0$. The relevant quantity is their largest value,
\[
 \lambda_\infty:=\sup\sigma(R_\infty)>0,
\]
which we call the \emph{defect strength}. Theorem~\ref{thm:bound-state} proves that the
finite-cycle spectral tops converge to this value.

\begin{figure}[t]
\centering
\begin{tikzpicture}[
  cv/.style={circle,draw,fill=black,inner sep=1.4pt},
  lf/.style={circle,draw,fill=white,inner sep=1.2pt}]
 \begin{scope}
   \foreach \k in {0,...,9} {\coordinate (p\k) at ({90-36*\k}:1.25);}
   \draw[thick] (p0)--(p1)--(p2)--(p3)--(p4)--(p5)--(p6)--(p7)--(p8)--(p9)--cycle;
   \foreach \k in {0,...,9} {\node[cv] at (p\k) {};}
   \node[lf] (h) at ($(p0)+(90:0.66)$) {}; \draw (p0)--(h);
   \node[lf] (ha) at ($(h)+(55:0.5)$) {}; \node[lf] (hb) at ($(h)+(125:0.5)$) {};
   \draw (h)--(ha); \draw (h)--(hb);
   \node[font=\scriptsize,below right] at (p0) {$v_0$};
   \node[font=\small] at (0,-1.75) {(a)\ $G_g$ (one defect)};
 \end{scope}
 \begin{scope}[shift={(5.9,-0.1)}]
   \foreach \k in {0,...,8} {\node[cv] (q\k) at ({(\k-4)*0.6},0) {};}
   \draw[thick] (q0)--(q8);
   \node[lf] (r) at (0,0.6) {}; \draw (q4)--(r);
   \node[lf] at (-0.34,1.12) {}; \node[lf] at (0.34,1.12) {};
   \draw (r)--(-0.34,1.12); \draw (r)--(0.34,1.12);
   \node[font=\scriptsize] at (-2.78,0) {$\cdots$};
   \node[font=\scriptsize] at (2.78,0) {$\cdots$};
   \node[font=\scriptsize,below] at (q4) {origin};
   \draw[blue!55,densely dashed] plot[smooth] coordinates
     {(-2.4,-0.12)(-1.8,-0.2)(-1.2,-0.33)(-0.6,-0.52)(0,-0.7)(0.6,-0.52)(1.2,-0.33)(1.8,-0.2)(2.4,-0.12)};
   \node[font=\scriptsize,blue!55!black] at (1.9,-0.62) {$\varphi_g\propto\mu^{-|j|}$};
   \node[font=\small] at (0,-1.75) {(b)\ $R_\infty$ (defect at origin)};
 \end{scope}
\end{tikzpicture}
\caption{A localized single-vertex defect: one rooted pendant tree on an otherwise bare
cycle $C_g$ (a), all other cycle vertices having degree $2$. As $g\to\infty$ the cycle unrolls to
the bi-infinite $2$-regular line carrying the same defect at the origin (b). The bare line has essential spectrum $[-2,0]$; the defect may produce finitely many
isolated positive eigenvalues. The dashed profile represents the positive eigenfunction at
the spectral top $\lambda_\infty$, which decays as $\mu^{-|j|}$ along the line
(Theorem~\ref{thm:bound-state} and Proposition~\ref{prop:threshold}).}
\label{fig:defect}
\end{figure}
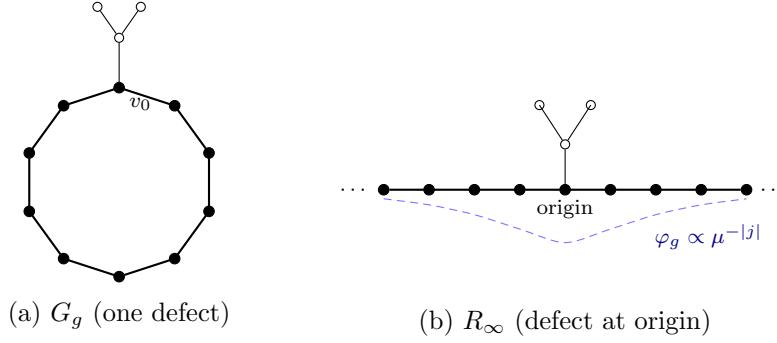

\begin{theorem}[Bound-state convergence]\label{thm:bound-state}
Let $G_g$ be the cycle $C_g$ carrying one fixed nonempty finite rooted pendant tree at a
single vertex $v_0$, and let $R_\infty$ be the operator obtained by attaching the same tree
to the origin of the bi-infinite $2$-regular line. Then
\[
 \sigma_{\mathrm{ess}}(R_\infty)=[-2,0],\qquad
 \lambda_\infty:=\sup\sigma(R_\infty)>0,
\]
$\lambda_\infty$ is an isolated simple eigenvalue, and
\[
 \lim_{g\to\infty}\lmax(R_{G_g})=\lambda_\infty.
\]
If $\mu>1$ is determined by
$\mu+\mu^{-1}=2(1+\lambda_\infty)$, a normalized eigenfunction at
$\lambda_\infty$ decays on each bare half-line as $O(\mu^{-|j|})$.
\end{theorem}
\begin{proof}
On the bare bi-infinite line the operator is
\[
 (R_0z)_j=-z_j+\tfrac12(z_{j-1}+z_{j+1}),
\]
whose Fourier symbol is $-1+\cos\theta$ and whose spectrum is $[-2,0]$. The attached finite
tree changes only finitely many matrix entries and adds finitely many coordinates. Thus
$R_\infty-R_0$ is finite rank (after adjoining a finite zero block), and Weyl's theorem gives
\[
 \sigma_{\mathrm{ess}}(R_\infty)=[-2,0].
\]

We next show that the spectral top is positive. Choose a leaf edge $e=\{u,\ell\}$ in the
attached tree and put $d=d_u$. Start with value $1$ on every edge in the finite defect and on
a long central segment of the line. Change the value on $e$ to $t$, and taper linearly from
$1$ to $0$ over $N$ line edges on each side. Only the local terms at $u$ and $\ell$ depend on
$t$. If their combined contribution in the vertex decomposition
\eqref{eq:vdecomp} is denoted by $H(t)$, then direct expansion gives
\[
 H(t)-H(1)=\frac{4}{d(d+1)}-\frac{d+1}{d}
 \left(t-\frac{d-1}{d+1}\right)^2.
\]
Here the finite defect means the attachment vertex together with the whole pendant tree,
and the degrees are those in the full infinite graph (so the attachment vertex also sees its
two line edges). For the all-ones vector its total contribution is
$\sum_{v\,\mathrm{in\ defect}}(d_v-2)=0$. Hence at
$t=(d-1)/(d+1)$ the defect gains $4/[d(d+1)]>0$. A degree-two line vertex with consecutive
edge values $x,y$ contributes $-\tfrac12(x-y)^2$, so the two linear tapers lose only
$O(N^{-1})$. For $N$ large the resulting finitely supported vector $\psi$ satisfies
$\langle\psi,R_\infty\psi\rangle>0$, and therefore
$\lambda_\infty=\sup\sigma(R_\infty)>0$.

Because $\lambda_\infty$ lies above the essential spectrum, it is an isolated eigenvalue of
finite multiplicity. Replacing an eigenfunction by its absolute value cannot decrease the
quadratic form, since every off-diagonal entry is nonnegative. Equality and connectedness of
the line graph force a top eigenfunction to have one strict sign. This also makes the top eigenspace one-dimensional. Indeed, if two independent positive top
eigenfunctions existed, a nontrivial linear combination could be chosen to vanish at one edge
without vanishing identically; it would still be a top eigenfunction, contradicting the strict-sign
conclusion. We choose the normalized top eigenfunction $\psi_\infty>0$.
Away from the defect it obeys
\[
 z_{j-1}+z_{j+1}=2(1+\lambda_\infty)z_j.
\]
The $\ell^2$ condition selects the decaying characteristic root, so
$|\psi_\infty(j)|\le C\mu^{-|j|}$, where
$\mu+\mu^{-1}=2(1+\lambda_\infty)$.

Let $\lambda_g=\lmax(R_{G_g})$. Truncate $\psi_\infty$ at distance
$\lfloor g/2\rfloor$ from the origin and transplant it to $G_g$. The two quadratic forms
agree except in a uniformly bounded neighbourhood of the antipodal cut, where the omitted
coordinates are $O(\mu^{-g/2})$. After normalization this gives
\begin{equation}\label{eq:defect-lower}
 \lambda_g\ge\lambda_\infty-C\mu^{-g}.
\end{equation}
In particular, for all large $g$, $\lambda_g\ge\lambda_\infty/2>0$.

Let $\varphi_g>0$ be the $\ell^2$-normalized Perron vector of $R_{G_g}$. Reflection of the
cycle fixing $v_0$ and the antipodal cut is a graph automorphism fixing the pendant tree, so
simplicity makes $\varphi_g$ reflection invariant. On the bare arc its coordinates solve the
same recurrence with
$\mu_g+\mu_g^{-1}=2(1+\lambda_g)$. By \eqref{eq:defect-lower},
$\mu_g\ge\mu_0:=\mu(\lambda_\infty/2)>1$ for all large $g$. If $A_g$ is the common value on
the two cycle edges at $v_0$, the solution on the bare arc is the hyperbolic interpolation
between the two boundary values $A_g$; hence, at distance $r$ from the nearer boundary,
\[
 |\varphi_g(c)|\le C A_g\mu_0^{-r}\le C\mu_0^{-r},
\]
with $C$ independent of $g$. In particular all coordinates in a fixed neighbourhood of the
antipodal cut are $O(\mu_0^{-g/2})$.

Cut the cycle in that neighbourhood, regard the remaining edge vector as a vector
$\widetilde\varphi_g$ on the infinite-line defect graph, and extend it by zero outside the
finite segment. Its squared norm is $1+O(\mu_0^{-g})$; after normalization, the two
quadratic forms differ only in the bounded cut neighbourhood. The preceding estimate therefore yields
\[
 \left|\langle\widetilde\varphi_g,R_\infty\widetilde\varphi_g\rangle-
 \lambda_g\right|\le C\mu_0^{-g}.
\]
The variational characterization of $\lambda_\infty$ gives
$\lambda_g\le\lambda_\infty+C\mu_0^{-g}$. Together with
\eqref{eq:defect-lower} this proves $\lambda_g\to\lambda_\infty$.
\end{proof}

\begin{proposition}\label{prop:threshold}
Let $G_g$ be a localized single-vertex defect: the cycle $C_g$ carrying one fixed
nonempty rooted pendant tree at a single vertex $v_0$, every other cycle vertex having
degree $2$. By Theorem~\ref{thm:bound-state},
$\lambda_g:=\lmax(R_{G_g})$ converges to a positive limit $\lambda_\infty$.
Let $\mu>1$ be the larger root of $\mu+\mu^{-1}=2(1+\lambda_\infty)$. Then the limiting
balance ratio is the closed form
\[
r_\infty:=\lim_{g\to\infty}\ \max_j\frac{w^\ast(c_{j-1})+w^\ast(c_j)+w^\ast(c_{j+1})}{L}
\ =\ g_{\!*}(\mu),\qquad
g_{\!*}(\mu):=\frac{(2\mu+1)(\mu-1)}{2\mu^2},
\]
and $g_{\!*}$ is strictly increasing with $g_{\!*}(\mu)=\tfrac12\iff\mu^2-\mu-1=0\iff
\mu=\varphi$ (the golden ratio). Equivalently, writing $\lambda_\infty=(\mu-1)^2/(2\mu)$,
\[
r_\infty\ \lessgtr\ \tfrac12
\qquad\Longleftrightarrow\qquad
\lambda_\infty\ \lessgtr\ \varphi-\tfrac32=\tfrac{\sqrt5-2}{2}\approx0.11803 .
\]
Consequently, if $\lambda_\infty<\varphi-\tfrac32$ then for all sufficiently large $g$ the
Perron metric $w^\ast$ is balanced \eqref{eq:3sum}, hence a discrete Einstein metric with
$\kappa=-\lmax(R_{G_g})$ (Theorem~\ref{thm:einstein}). If
$\lambda_\infty>\varphi-\tfrac32$, the Perron metric is unbalanced for all sufficiently
large $g$. No finite-$g$ conclusion follows from the limiting ratio alone in the equality
case. The operative quantity is thus the
local defect strength $\lambda_\infty$, not the global Perron spread $\rho_G$ of
Proposition~\ref{prop:apriori}.
\end{proposition}
\begin{proof}
Let the cycle edges be $c_1,\dots,c_g$ with $c_1,c_g$ the two edges at $v_0$. The
reflection $v_k\mapsto v_{g-k}$ fixes the pendant tree pointwise and reverses the cycle, so
it is an automorphism of $G_g$; as $\RG$ depends only on degrees and incidence, its
permutation matrix commutes with $\RG$. The Perron vector $w^\ast$ is unique and positive
(Lemma~\ref{lem:pf}), hence fixed by the reflection: $w^\ast(c_j)=w^\ast(c_{g+1-j})$, and in
particular $w^\ast(c_1)=w^\ast(c_g)$, a symmetric doublet at $v_0$. Index the
cycle edges $j=1,2,\dots$ outward from $c_1$. At any cycle edge with both endpoints of
degree $2$ the eigen-equation is
$-w_j+\tfrac12(w_{j-1}+w_{j+1})=\lambda_g w_j$, i.e.\ the
homogeneous recurrence
$w_{j-1}+w_{j+1}=2(1+\lambda_g)w_j$ (the $d=2,\ \Psi=0$ case of
\eqref{eq:recur}), with characteristic roots $\mu_g^{\pm1}$, where
$\mu_g+\mu_g^{-1}=2(1+\lambda_g)$. By the reflection symmetry the solution is the
even combination
$w_j\propto\cosh\!\big(\kappa_g(j-\tfrac{g+1}2)\big)$,
$\mu_g=e^{\kappa_g}$. Theorem~\ref{thm:bound-state} gives $\lambda_g\to\lambda_\infty>0$ and hence
$\mu_g\to\mu>1$; hence, for each fixed $j$,
$w_j/w_1\to\mu^{-(j-1)}$. The same explicit cosh formula gives a uniform geometric
majorant because $\mu_g$ is bounded below by a constant greater than $1$ for large
$g$. Therefore the finite perimeter sums may be passed to the limit.

Since $\cosh$ decreases monotonically from $v_0$ to the antipode, the maximal
three-edge window is $\{c_g,c_1,c_2\}$. Let $L_g$ be the cycle perimeter and write
$w_{\mathrm p,g}:=w^\ast(c_1)=w^\ast(c_g)$. The window weight divided by
$w_{\mathrm p,g}$ tends to $2+\mu^{-1}$, while the normalized perimeter tends to
the doublet plus two geometric tails,
\[
\lim_{g\to\infty}\frac{L_g}{w_{\mathrm p,g}}
=2+2\sum_{k\ge1}\mu^{-k}
=2\,\frac{\mu}{\mu-1}
\qquad\Big(\textstyle\sum_{k\ge1}\mu^{-k}=\tfrac1{\mu-1}\Big).
\]
Hence $r_\infty=(2+\mu^{-1})(\mu-1)/(2\mu)=(2\mu+1)(\mu-1)/(2\mu^2)=g_{\!*}(\mu)$. Finally
$g_{\!*}(\mu)=1-\tfrac1{2\mu}-\tfrac1{2\mu^2}$ has $g_{\!*}'(\mu)=\tfrac1{2\mu^2}+\tfrac1{\mu^3}>0$,
and $g_{\!*}(\mu)=\tfrac12$ reads $\mu^2-\mu-1=0$, i.e.\ $\mu=\varphi$; from
$\lambda=(\mu-1)^2/(2\mu)$ this is $\lambda=\varphi-\tfrac32$. The balance criterion
\eqref{eq:3sum} is $r\le\tfrac12$. Thus
$\lambda_\infty<\varphi-\tfrac32$ gives $r(g)<\tfrac12$ for all large $g$ and
Theorem~\ref{thm:einstein} makes $w^\ast$ Einstein; the reversed strict inequality
gives $r(g)>\tfrac12$ for all large $g$. When equality holds, convergence
$r(g)\to\tfrac12$ alone does not determine on which side the finite values lie.
\end{proof}

The appearance of the golden ratio has a direct geometric-series interpretation. Normalise
the weight at $v_0$ to $1$, so
that along each side of the cycle the Perron weights are (in the limit) the geometric
sequence $1,\mu^{-1},\mu^{-2},\dots$. In the balance condition \eqref{eq:3sum} at $v_0$,
the three-edge window and half the perimeter share their first terms, and cancelling them
leaves exactly
\[
1\ \lessgtr\ \mu^{-2}+\mu^{-3}+\cdots=\frac{1}{\mu(\mu-1)}:
\]
the weight at $v_0$ against the total weight sitting two or more steps away on one side.
Equality is $\mu(\mu-1)=1$, i.e.\ $\mu^2=\mu+1$, the defining identity of $\varphi$.

\begin{remark}\label{rem:threshold-exist}
Proposition~\ref{prop:threshold} concerns balance --- whether the fixed local matrix $R_G$
computes the curvature --- not existence. Above the threshold
($\lambda_\infty>\varphi-\tfrac32$) its Perron vector is unbalanced and therefore is not an
 Einstein metric with curvature $-\lmax(R_G)$. Whether a different Einstein metric exists in
 the corresponding full transport problem is a separate question.
\end{remark}

\begin{conjecture}\label{conj:general-threshold}
Let $G_g$ be obtained from $C_g$ by inserting a fixed finite decoration in an arc of uniformly
bounded length, with the remaining cycle vertices of degree $2$. Suppose that
$\lambda_\infty:=\lim_{g\to\infty}\lmax(R_{G_g})$ exists and satisfies
$\lambda_\infty<\varphi-\tfrac32$. Then the Perron weighting of $R_{G_g}$ is balanced, and
hence Einstein, for all sufficiently large $g$.
\end{conjecture}
For several defects the limiting balance ratio need not be controlled by
$\lambda_\infty$ alone. A possible route to the conjecture is to estimate the interaction
of the localized Perron profiles uniformly below the margin
$\tfrac12-g_{\!*}(\mu)$.

\section{Short cycles}\label{sec:short}

For the uniform weighting and $g\le5$, three consecutive edges have relative length
$3/g>1/2$, so every cycle edge is unbalanced. For non-uniform weights the conclusion is
not automatic: on $C_4$ or $C_5$ a three-edge window can occupy at most half of the
perimeter, and the corresponding edge then obeys the tree formula by
Proposition~\ref{prop:balanced}. We now compute the curvature in the complementary
geodesic regime containing the uniform weighting. Fix $C_g$ ($g\le5$) with edges
$c_0,\dots,c_{g-1}$, $c_j=\{j,j{+}1\}$ of weight $w_j$ (indices mod $g$); the two outer
cycle neighbours of $c_j$ are the vertices $j{-}1$ and $j{+}2$.

\begin{proposition}\label{prop:short}
Suppose the edge $c_j$ and the arc from $j{-}1$ to $j{+}2$ avoiding $c_j$ are
geodesics between their endpoints (this holds on a neighbourhood of the uniform
weighting). Then
\[
\kappa_{c_j}=\frac12\Big(3-\frac{d(j{-}1,\,j{+}2)}{w_j}\Big),
\]
$d(j{-}1,j{+}2)$ being the length of that avoiding arc. Explicitly,
\[
C_3:\ \kappa_{c_j}\equiv\tfrac32;\qquad
C_4:\ \kappa_{c_j}=\tfrac12\Big(3-\tfrac{w_{j+2}}{w_j}\Big);\qquad
C_5:\ \kappa_{c_j}=\tfrac12\Big(3-\tfrac{w_{j+2}+w_{j+3}}{w_j}\Big).
\]
In particular $\kappa_{c_j}w_j=\tfrac12\big(3w_j-d(j{-}1,j{+}2)\big)$ is linear in $w$.
If the two geodesic assumptions hold for every $j$, they define the short-cycle regime
$\sigma_{\mathrm{sc}}$ containing the uniform weighting, and its active operator is defined by
$(\kappa_{c_j}w_j)_j=-R_{C_g}^{(\sigma_{\mathrm{sc}})}w$. In that regime
\[
R_{C_g}^{(\sigma_{\mathrm{sc}})}=
\begin{cases}
-\tfrac32 I, & g=3,\\
-\tfrac32 I+\tfrac12 P_{\mathrm{opp}}, & g=4,\\
-\tfrac32 I+\tfrac12 P_{2}, & g=5,
\end{cases}
\]
where $P_{\mathrm{opp}}$ couples opposite edges and $P_2$ couples edges at line-graph
distance $2$: explicitly, in the edge bases $(c_0,c_1,c_2,c_3)$ and $(c_0,\dots,c_4)$,
\[
P_{\mathrm{opp}}=\begin{pmatrix}0&0&1&0\\0&0&0&1\\1&0&0&0\\0&1&0&0\end{pmatrix},
\qquad
P_{2}=\begin{pmatrix}0&0&1&1&0\\0&0&0&1&1\\1&0&0&0&1\\1&1&0&0&0\\0&1&1&0&0\end{pmatrix},
\]
i.e.\ $(P_{\mathrm{opp}})_{jk}=1$ iff $k=j+2\ (\mathrm{mod}\ 4)$ and $(P_2)_{jk}=1$ iff
$k=j\pm2\ (\mathrm{mod}\ 5)$. The uniform metric is Einstein with
$\kappa=\tfrac32,\,1,\,\tfrac12$ respectively (Figure~\ref{fig:short}).
\end{proposition}

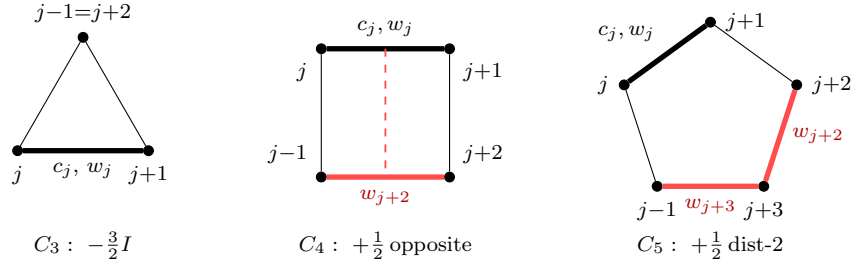
\begin{figure}[t]
\centering
\begin{tikzpicture}[scale=1.0, cv/.style={circle,draw,fill=black,inner sep=1.3pt},
   wl/.style={font=\scriptsize}, vl/.style={font=\scriptsize}]
 \begin{scope}[shift={(0,0)}]
   \node[cv,label={[vl]below:$j$}] (a0) at (210:1) {};
   \node[cv,label={[vl]below:$j{+}1$}] (a1) at (330:1) {};
   \node[cv,label={[vl]above:$j{-}1{=}j{+}2$}] (a2) at (90:1) {};
   \draw[line width=2pt] (a0)--(a1) node[wl,midway,below]{$c_j,\,w_j$};
   \draw (a1)--(a2); \draw (a2)--(a0);
   \node[font=\scriptsize] at (0,-1.75) {$C_3:\ -\tfrac32 I$};
 \end{scope}
 \begin{scope}[shift={(4.0,0)}]
   \node[cv,label={[vl]below left:$j$}]   (b0) at (135:1.2) {};
   \node[cv,label={[vl]below right:$j{+}1$}] (b1) at (45:1.2) {};
   \node[cv,label={[vl]above right:$j{+}2$}] (b2) at (315:1.2) {};
   \node[cv,label={[vl]above left:$j{-}1$}]  (b3) at (225:1.2) {};
   \draw[line width=2pt] (b0)--(b1) node[wl,midway,above]{$c_j,w_j$};
   \draw (b1)--(b2);
   \draw[line width=2pt,red!70] (b2)--(b3) node[wl,midway,below,red!70!black]{$w_{j+2}$};
   \draw (b3)--(b0);
   \draw[red!70,dashed,line width=0.7pt] ($(b0)!0.5!(b1)$)--($(b2)!0.5!(b3)$);
   \node[font=\scriptsize] at (0,-1.75) {$C_4:\ +\tfrac12\,$opposite};
 \end{scope}
 \begin{scope}[shift={(8.3,0)}]
   \node[cv,label={[vl]left:$j$}]      (c0) at (162:1.2) {};
   \node[cv,label={[vl]right:$j{+}1$}] (c1) at (90:1.2) {};
   \node[cv,label={[vl]right:$j{+}2$}] (c2) at (18:1.2) {};
   \node[cv,label={[vl]below:$j{+}3$}] (c3) at (306:1.2) {};
   \node[cv,label={[vl]below:$j{-}1$}] (c4) at (234:1.2) {};
   \draw[line width=2pt] (c0)--(c1) node[wl,midway,above left]{$c_j,w_j$};
   \draw (c1)--(c2);
   \draw[line width=2pt,red!70] (c2)--(c3) node[wl,midway,right,red!70!black]{$w_{j+2}$};
   \draw[line width=2pt,red!70] (c3)--(c4) node[wl,midway,below,red!70!black]{$w_{j+3}$};
   \draw (c4)--(c0);
   \node[font=\scriptsize] at (0,-1.75) {$C_5:\ +\tfrac12\,$dist-$2$};
 \end{scope}
\end{tikzpicture}
\caption{The short-cycle transport-regime operators $R_{C_g}^{(\sigma_{\mathrm{sc}})}$
(Proposition~\ref{prop:short}), with vertices indexed as in the text. The curvature of the direct edge $c_j=\{j,j{+}1\}$ of weight $w_j$ (thick) is set by the
avoiding arc from $j{-}1$ to $j{+}2$, and couples --- through the $+\tfrac12$
off-diagonal --- to the edges drawn in red: none for $C_3$ (operator $-\tfrac32 I$; there
$j{-}1=j{+}2$ and the arc has length $0$), the opposite edge $w_{j+2}$ for $C_4$, and the two
line-graph-distance-$2$ edges $w_{j+2},w_{j+3}$ for $C_5$.}
\label{fig:short}
\end{figure}

\begin{proof}
Both endpoints of $c_j$ have degree $2$, so $\mu_x^\alpha$ places $\alpha$ on $x=j$ and
$\tfrac{1-\alpha}2$ on each of $j{-}1,j{+}1$, and similarly for $y=j{+}1$. The signed
measure $\mu_x^\alpha-\mu_y^\alpha$ has surplus $A=\tfrac{3\alpha-1}2$ at $x$ and
$B=\tfrac{1-\alpha}2$ at the outer neighbour $x'=j{-}1$, and matching deficits $A$ at
$y=j{+}1$ and $B$ at $y'=j{+}2$; the mass on the shared neighbourhood cancels, and for
$C_3$ the points $x',y'$ coincide so $B$ cancels entirely. In the stated geodesic regime
the optimal plan moves $A$ along $c_j$ (distance $w_j$) and $B$ from $x'$ to $y'$
(distance $d(x',y')$). To verify optimality, put
$p=w_{j-1}$, $a=w_j$, $q=w_{j+1}$, and $D=d(x',y')$, the length of the
complementary arc. The two cross-distances are
\[
 d(x,y')=\min(a+q,p+D),\qquad d(x',y)=\min(p+a,D+q).
\]
The assumptions that the direct edge and the complementary arc are geodesics say
$a\le p+D+q$ and $D\le p+a+q$. Checking the four possible choices in the two minima
therefore gives the Monge inequality
\[
   a+D\le d(x,y')+d(x',y).
\]
For $\alpha$ sufficiently close to $1$ one has $A\ge B$. Every coupling of the two
surpluses $(A,B)$ to the two deficits $(A,B)$ is obtained from the proposed coupling
by sending some $s\in[0,B]$ along each cross-pair instead; its cost exceeds the
proposed cost by
$s\,[d(x,y')+d(x',y)-a-D]\ge0$. Hence the proposed coupling is optimal and
\[
W(\mu_x^\alpha,\mu_y^\alpha)=A\,w_j+B\,d(x',y').
\]
With
$d(x,y)=w_j$,
\[
\kappa_\alpha=1-\frac{A\,w_j+B\,d(x',y')}{w_j}.
\]
Setting $\alpha=1-\varepsilon$ gives $A=1-\tfrac{3\varepsilon}2$, $B=\tfrac\varepsilon2$,
whence, exactly,
$\kappa_\alpha=\tfrac{\varepsilon}2\big(3-d(x',y')/w_j\big)$ and
$\kappa_{c_j}=\lim_{\varepsilon\to0}\kappa_\alpha/\varepsilon=\tfrac12(3-d(x',y')/w_j)$.
The three cases substitute $d(x',y')=0,\ w_{j+2},\ w_{j+2}+w_{j+3}$, and multiplying by
$w_j$ gives the linear expressions and hence the regime operators.
\end{proof}

The $C_4$ transport-regime operator $-\tfrac32 I+\tfrac12 P_{\mathrm{opp}}$ is reducible:
$P_{\mathrm{opp}}$ splits the four edges into two opposite pairs, so its largest
eigenvalue $-1$ (giving $\kappa=1$) has a two-dimensional eigenspace and the Perron vector
contains the positive vectors $(a,b,a,b)$ with $a,b>0$. Thus Perron uniqueness fails for
$C_4$, which cannot happen on a tree, where $R_T+2I$ is always irreducible. For the
uniform weighting, as the girth grows the coupling migrates outward along the line graph
(distances $0,2,2,\dots$) and, by $g\ge6$, becomes the adjacent-edge (tree-formula)
coupling of Proposition~\ref{prop:bare}, re-entering the balanced regime. A description
of the full shortest-path and transport-cell structure away from the regime of
Proposition~\ref{prop:short} is left to future work.

\section{Existence of Einstein metrics}\label{sec:exist}

If the Perron vector of $\RG$ is balanced, it is an Einstein metric with
$\kappa=-\lmax(\RG)$ (Theorem~\ref{thm:einstein}). Outside the balanced regime no single
matrix describes the whole positive cone: on a linear geodesic transport cell the equation
is instead $R_G^{(\sigma)}w=-\kappa w$, and the cell data may be finer than the
balanced/wrap pattern (Remark~\ref{rem:transport-scope}). Balance is sufficient but not
necessary for existence: the uniformly weighted $C_4$ and $C_5$, and the symmetric metrics
on $S_{g,d}$ for $g\in\{3,4,5\}$, are Einstein and unbalanced. Nor is existence automatic,
as the following example shows.

\begin{example}\label{thm:nonexist}
The unicyclic graph $C_3^{+}$ obtained from a triangle by attaching a single pendant leaf
admits no discrete Einstein metric.
\end{example}
\begin{proof}
Label the triangle $\{0,1,2\}$ and let $3$ be the leaf, attached to $0$; write
$w_{01},w_{12},w_{02},w_{03}$ for the edge weights and $d$ for the graph metric. We bound
two curvatures unconditionally --- with no assumption on the weights --- and show
they can never coincide.

\emph{The cycle edge: $\kappa_{12}=\tfrac32$ always.} The endpoints $1,2$ have degree
$2$ and share the common neighbour $0$, so the mass at $0$ cancels in
$\rho=\mu_1^\alpha-\mu_2^\alpha$: $\rho(0)=\tfrac{1-\alpha}2-\tfrac{1-\alpha}2=0$. Thus
$\rho$ is a single surplus $\tfrac{3\alpha-1}2$ at $1$ and an equal deficit at $2$, and its
optimal transport simply moves that mass from $1$ to $2$, giving
$W(\mu_1^\alpha,\mu_2^\alpha)=\tfrac{3\alpha-1}2\,d(1,2)$ and
\[
\kappa_\alpha(1,2)=1-\frac{W}{d(1,2)}=1-\frac{3\alpha-1}2 ,
\]
in which $d(1,2)$ cancels. Hence $\kappa_{12}=\lim_{\alpha\to1}\kappa_\alpha/(1-\alpha)=\tfrac32$
for every weighting --- whether or not $\{1,2\}$ is a geodesic, and independently of the
leaf weight.

\emph{The leaf edge: $\kappa_{03}<\tfrac43$ always.} Set $f(0)=0$, $f(3)=-w_{03}$,
$f(1)=d(0,1)$, $f(2)=d(0,2)$. This $f$ is $1$-Lipschitz for $d$: indeed
$|f(1)-f(2)|=|d(0,1)-d(0,2)|\le d(1,2)$, and $|f(i)-f(3)|=d(0,i)+w_{03}=d(i,3)$ since $3$ is
a leaf at $0$. A direct evaluation of $\langle f,\mu_0^\alpha-\mu_3^\alpha\rangle$ yields
\[
\kappa_{03}\ \le\ \frac43-\frac{d(0,1)+d(0,2)}{3\,w_{03}}\ <\ \frac43 ,
\]
since $d(0,1),d(0,2)>0$ (the value $\tfrac43$ is approached only as $w_{03}\to\infty$; a
small leaf weight makes $\kappa_{03}$ more negative, not larger).

Therefore $\kappa_{12}-\kappa_{03}>\tfrac32-\tfrac43=\tfrac16>0$ for every positive
weighting, so the curvature is never constant, and $C_3^{+}$ has no Einstein metric.
\end{proof}

Example~\ref{thm:nonexist} contrasts sharply with trees, each of which carries a
unique Einstein metric \cite{BCH}. For the edge opposite the leaf, both endpoints have degree
two and the masses at their common neighbour cancel, fixing its curvature at $\tfrac32$;
the adjacent branch cannot match it. With four edges, $C_3^{+}$ is the smallest unicyclic
graph with no Einstein metric.

Symmetry, by contrast, forces existence independently of the balanced condition.

\begin{example}\label{ex:barecycles}
If $\mathrm{Aut}(G)$ acts transitively on $E(G)$, then the curvature of the uniform weight
$w\equiv1$ takes one value on the single edge orbit, so $w\equiv1$ is Einstein. Among
unicyclic graphs the edge-transitive ones are exactly the bare cycles. Every bare cycle is therefore Einstein, with
curvature $\tfrac32,\,1,\,\tfrac12$ for $g=3,4,5$ (Section~\ref{sec:short}) and $0$ for
$g\ge6$ (Proposition~\ref{prop:bare}).
\end{example}

For regular suns the two edge orbits reduce the Einstein equation to one explicit scalar
equation.

\begin{proposition}\label{prop:sununiq}
Normalize the common leaf weight of $S_{g,d}$ to $1$ and write $a>0$ for the common
cycle-edge weight. Then
\[
 \kappa_{\mathrm{leaf}}(a)=\frac{4-2a}{d},\qquad
 \kappa_{\mathrm{cyc}}(a)=-\frac{2(d-2)}{ad}+\frac{(6-g)^+}{d},
 \qquad (6-g)^+:=\max\{6-g,0\}.
\]
The symmetric Einstein equation has exactly one positive solution,
\[
 a^*=\begin{cases}
 1+\sqrt{d-1},&g\ge6,\\[4pt]
 \dfrac{g-2+\sqrt{(g-2)^2+16(d-2)}}4,&3\le g\le5.
 \end{cases}
\]
Thus the symmetric Einstein metric is unique up to scaling, with curvature
$\kappa=(4-2a^*)/d$.
\end{proposition}
\begin{proof}
The leaf formula follows from Lemma~\ref{lem:leaf-rigidity} below: at a cycle vertex the sum
of the graph distances to the two cycle neighbours is $2a$, so
$\kappa_{\mathrm{leaf}}=(4-2a)/d$.

For a cycle edge $c=\{x,y\}$, the signed measure
$\mu_x^\alpha-\mu_y^\alpha$ has core mass $\alpha-\varepsilon/d$ across $c$, outer masses
$\varepsilon/d$, and $d-2$ leaf masses on each side. If $g\ge6$, the three-edge through arc
is geodesic and the tree-formula computation gives
$\kappa_{\mathrm{cyc}}=-2(d-2)/(ad)$. If $3\le g\le5$, the outer pair is instead joined by
the wrap arc of length $(g-3)a$ rather than the through length $3a$. The optimal transport
cost is therefore lower by $\varepsilon(6-g)a/d$, so the curvature is higher by
$(6-g)/d$. This proves the displayed formula (the same computation is recorded locally in
Proposition~\ref{prop:short}).

Equating the two curvatures gives
\[
 2a^2+\bigl((6-g)^+-4\bigr)a-2(d-2)=0.
\]
Its constant term is negative, so it has exactly one positive root; the two displayed formulas
follow by the quadratic formula. Equivalently,
$\Delta(a)=\kappa_{\mathrm{cyc}}(a)-\kappa_{\mathrm{leaf}}(a)$ satisfies
$\Delta'(a)=2/d+2(d-2)/(a^2d)>0$, which also proves uniqueness.
\end{proof}

\begin{corollary}\label{cor:symexist}
Every regular sun $S_{g,d}$ $(g\ge3,d\ge3)$ admits a symmetric discrete Einstein metric.
For $g\ge6$ it is the balanced Perron metric of Theorem~\ref{thm:sun}; for $3\le g\le5$
it is unbalanced.
\end{corollary}
\begin{proof}
Use the positive root $a^*$ of Proposition~\ref{prop:sununiq}. For $g\ge6$, uniform cycle
weights satisfy the balanced condition; for $g\le5$, $3a^*>ga^*/2$, so every cycle edge is
in the wrap regime.
\end{proof}

\begin{lemma}[Leaf rigidity on a sun]\label{lem:leaf-rigidity}
Let $w$ be any positive weighting of $S_{g,d}$, with no geodesicity assumption. At a cycle
vertex $v$, let $b_1,\ldots,b_{d-2}$ be the incident leaf-edge weights, put
$B_v=\sum_i b_i$, and let
\[
 \delta_v=d(v,v_-)+d(v,v_+)
\]
be the sum of the graph distances from $v$ to its two cycle neighbours. Then
\[
 \kappa_{v\ell_i}=1+\frac2d-\frac{\delta_v+B_v}{d\,b_i}.
\]
Consequently, in any Einstein metric the sibling leaves at $v$ have one common weight
$b_v$, and
\[
 \delta_v=(4-\kappa d)b_v.
\]
In particular $\rho:=4-\kappa d>0$ and $b_v=\delta_v/\rho$.
\end{lemma}
\begin{proof}
Put $\varepsilon=1-\alpha$. For the leaf edge $\{v,\ell_i\}$, send the surplus at $v$ and
at every other neighbour of $v$ to $\ell_i$ along a shortest path, leaving the common masses
on the diagonal. The distance-to-$\ell_i$ function is a matching Kantorovich potential, so
the plan is optimal. Its cost is
\[
 W=\alpha b_i+\frac{\varepsilon}{d}\bigl(\delta_v+B_v-2b_i\bigr).
\]
Dividing by $b_i=d(v,\ell_i)$ and taking the Lin--Lu--Yau slope gives the formula. If all
leaf curvatures at $v$ equal the Einstein value, then
$(\delta_v+B_v)(b_i^{-1}-b_j^{-1})=0$. Since $\delta_v+B_v>0$, all $b_i$ are equal.
Substituting $B_v=(d-2)b_v$ yields
$\kappa=4/d-\delta_v/(db_v)$, equivalently
$\delta_v=(4-\kappa d)b_v$.
\end{proof}

\begin{lemma}[Curvature formulas for a triangle sun]\label{lem:triangle-sun-curvature}
Let $G=S_{3,d}$ with cycle vertices $v_0,v_1,v_2$, where $d\ge3$. Put
\[
 D_{ij}:=d(v_i,v_j),\qquad
 B_i:=\sum_{\substack{\ell\sim v_i\\ \ell\text{ a leaf}}}w_{v_i\ell}.
\]
Then, for every distinct $i,j\in\{0,1,2\}$,
\begin{equation}\label{eq:triangle-cycle-curvature}
 \kappa_{v_iv_j}
 =\frac3d-\frac{B_i+B_j}{dD_{ij}}.
\end{equation}
If the sibling leaves at $v_i$ have a common weight $b_i$, as they do in every
Einstein metric by Lemma~\ref{lem:leaf-rigidity}, then for
$\{i,j,k\}=\{0,1,2\}$,
\begin{equation}\label{eq:triangle-sun-curvatures}
 \kappa_{v_iv_j}
 =\frac3d-\frac{(d-2)(b_i+b_j)}{dD_{ij}},
 \qquad
 \kappa_{v_i\ell}
 =\frac4d-\frac{D_{ij}+D_{ik}}{db_i}.
\end{equation}
These identities hold for every positive weighting; no triangle inequality or
geodesicity assumption on the three cycle-edge weights is required.
\end{lemma}
\begin{proof}
Fix the cycle edge $\{v_i,v_j\}$ and let $v_k$ be the third cycle vertex. For
$\alpha$ sufficiently close to $1$, write $\varepsilon=1-\alpha$ and
$m=\varepsilon/d$. In the signed measure
$\mu_{v_i}^\alpha-\mu_{v_j}^\alpha$, the masses at $v_k$ cancel. The remaining
positive masses are
\[
 \alpha-m\quad\text{at }v_i,
 \qquad m\quad\text{at each leaf of }v_i,
\]
and the corresponding negative masses occur at $v_j$ and at the leaves of $v_j$.
By the common-mass cancellation property of $W_1$, it remains to couple precisely these
positive and negative parts.
For $x\in\{v_i\}\cup L_i$ and $y\in\{v_j\}\cup L_j$, where $L_i,L_j$ are the two
leaf sets, pendantness gives the additive distance formula
\[
 d(x,y)=r_i(x)+D_{ij}+r_j(y),
\]
with $r_i(v_i)=0$, $r_i(\ell)=w_{v_i\ell}$, and similarly at $v_j$. Hence the
cost of every coupling between the positive and negative parts is the same. Since
the total mass crossing from the $i$-side to the $j$-side is
\[
 (\alpha-m)+(d-2)m=1-\frac{3\varepsilon}{d},
\]
we obtain exactly
\[
 W(\mu_{v_i}^\alpha,\mu_{v_j}^\alpha)
 =\left(1-\frac{3\varepsilon}{d}\right)D_{ij}
  +\frac{\varepsilon}{d}(B_i+B_j).
\]
Dividing by $D_{ij}$ and taking the Lin--Lu--Yau slope yields
\eqref{eq:triangle-cycle-curvature}. If the sibling leaves are equal, then
$B_i=(d-2)b_i$ and the first formula in \eqref{eq:triangle-sun-curvatures} follows.
The second is Lemma~\ref{lem:leaf-rigidity}, because
$\delta_{v_i}=D_{ij}+D_{ik}$.
\end{proof}

\begin{proposition}[Triangle-sun rigidity]\label{prop:triangle-sun}
For every $d\ge3$, the triangle sun $S_{3,d}$ carries exactly one discrete
Einstein metric up to scaling. It is the symmetric metric of
Proposition~\ref{prop:sununiq}; after normalising the common leaf weight to $1$,
its common cycle-edge weight is
\[
 a^*=\frac{1+\sqrt{1+16(d-2)}}4.
\]
\end{proposition}
\begin{proof}
Let $w$ be an Einstein metric of curvature $\kappa$. By
Lemma~\ref{lem:leaf-rigidity}, the sibling leaves at $v_i$ have a common weight
$b_i$. Put
\[
 C:=4-\kappa d,\qquad K:=3-\kappa d=C-1,\qquad n:=d-2,
\]
and set
\[
 p:=D_{12},\qquad q:=D_{20},\qquad r:=D_{01}.
\]
The leaf and cycle formulas of Lemma~\ref{lem:triangle-sun-curvature} give
\[
 Cb_0=q+r,\qquad Cb_1=r+p,\qquad Cb_2=p+q,
\]
and
\[
 Kp=n(b_1+b_2),\qquad Kq=n(b_2+b_0),\qquad Kr=n(b_0+b_1).
\]
In particular $K>0$. Eliminating the $b_i$ gives
\[
 KCp=n(2p+q+r),\quad
 KCq=n(p+2q+r),\quad
 KCr=n(p+q+2r).
\]
Subtracting two consecutive equations yields
\[
 (p-q)(KC-n)=0,\qquad (q-r)(KC-n)=0.
\]
If $KC=n$, the first displayed equation reduces to
$p=2p+q+r$, impossible because $p,q,r>0$. Hence $p=q=r=:a$.

We next recover the actual cycle-edge weights from their metric distances. If, for
example, the edge joining $v_1$ to $v_2$ had weight strictly larger than $D_{12}=a$,
then its complementary two-edge route would realise $D_{12}$, and therefore
\[
 a=D_{12}=w_{20}+w_{01}\ge D_{20}+D_{01}=2a,
\]
a contradiction. Thus each cycle edge is itself geodesic and has weight $a$; the
cycle weighting is symmetric. The leaf equations then give
$b_0=b_1=b_2=2a/C$.

Finally the cycle equation becomes $KCa=4na$, so
\[
 K(K+1)=4(d-2).
\]
It has exactly one positive root. Since $C=K+1$ and $b=2a/C$, normalising $b=1$
gives
\[
 a=\frac C2=\frac{1+\sqrt{1+16(d-2)}}4.
\]
This is precisely the symmetric metric of Proposition~\ref{prop:sununiq}, and proves
uniqueness among all positive weightings, including the non-geodesic cones.
\end{proof}

\begin{proposition}\label{prop:sunrigid}
Let $G=S_{g,d}$, and let $w$ be a geodesic discrete Einstein metric of curvature
$\kappa$ on $G$.
\begin{enumerate}
\item[\rm(i)] The leaves at each cycle vertex $v$ carry equal weight $b_v$. If the two
incident cycle-edge weights are $a,a'$, then geodesicity gives
$a+a'=(4-\kappa d)b_v$.
\item[\rm(ii)] Put
$x=(a_0,\ldots,a_{g-1},b_0,\ldots,b_{g-1})^{\mathsf T}$ and record by
$\sigma\in\{B,W\}^g$ whether each cycle edge is in the balanced or wrap regime.  The
Einstein equations in this regime have the form
$N_\sigma x=(\kappa d)x$ for a fixed real matrix $N_\sigma$.  Although $N_\sigma$ need
not be symmetric in a mixed regime, $4I-N_\sigma$ is nonnegative and irreducible.
Consequently $N_\sigma$ has a unique positive eigenvector up to scaling, and there is
at most one Einstein metric in any fixed self-consistent regime.
\end{enumerate}
On the balanced cone, including its boundary, the corresponding uniqueness conclusion is
Theorem~\ref{thm:einstein}.
\end{proposition}
\begin{proof}
(i) Lemma~\ref{lem:leaf-rigidity} gives equal sibling weights and
$\delta_v=(4-\kappa d)b_v$. Under the geodesic hypothesis the two distances from $v$ to its
cycle neighbours are the two incident edge weights, so $\delta_v=a+a'$.

(ii) Indices below are taken modulo $g$, and $L=\sum_k a_k$. For $g\ge4$ the cycle-edge
calculation is Lemma~\ref{lem:fmax} below; for $g=3$ every geodesic cycle edge is of wrap
type, and the same wrap row follows from
Lemma~\ref{lem:triangle-sun-curvature} (where $D_{j,j+1}=a_j$).
Together with the leaf equation in (i), this gives
the following rows of the reduced regime matrix:
\[
(N_\sigma x)_j=
\begin{cases}
2a_j-a_{j-1}-a_{j+1}-(d-2)(b_j+b_{j+1}),&\sigma_j=B,\\[2pt]
4a_j+a_{j-1}+a_{j+1}-L-(d-2)(b_j+b_{j+1}),&\sigma_j=W,
\end{cases}
\]
and
\[
(N_\sigma x)_{g+v}=4b_v-a_{v-1}-a_v.
\]
Thus $N_\sigma x=(\kappa d)x$ is precisely the reduced Einstein system.  After the
$-L$ term is expanded, every off-diagonal entry of $N_\sigma$ is nonpositive and every
diagonal entry is at most $4$.  Hence $M_\sigma:=4I-N_\sigma$ is nonnegative.  Moreover
its directed support contains the two-way links
$a_j\leftrightarrow b_j$ and $a_j\leftrightarrow b_{j+1}$ for every $j$; because
$d-2>0$, these links connect all $2g$ coordinates, so $M_\sigma$ is irreducible.
If $x>0$ is an Einstein vector and $c=\kappa d$, then
$M_\sigma x=(4-c)x$.  Perron--Frobenius therefore gives
$4-c=\rho(M_\sigma)$ and makes $x$ the unique positive eigenvector up to scaling.
\end{proof}

\begin{proposition}\label{prop:sunbalsym}
Any balanced discrete Einstein metric on a sun $S_{g,d}$ is symmetric --- all cycle
edges carry one weight and all leaves another --- and hence equals the metric of
Proposition~\ref{prop:sununiq}, unique up to scaling.
\end{proposition}
\begin{proof}
For $g=3$ there is no balanced weighting by Remark~\ref{rem:g3}, so assume $g\ge4$.
By Theorem~\ref{thm:einstein}, $w$ is proportional to the positive Perron eigenvector of
$\RG$. Every automorphism of $S_{g,d}$ induces a permutation of the edge coordinates that
commutes with $\RG$, because the entries in \eqref{eq:RT} depend only on incidence and
vertex degrees. After normalizing the Perron vector to have coordinate sum $1$, its
simplicity therefore makes it invariant under every automorphism. The automorphism group is
transitive on the cycle edges and on the leaf edges separately: rotations move the cycle
vertices and permutations of the leaves at each vertex move the sibling leaves. Hence all
cycle-edge coordinates are equal and all leaf-edge coordinates are equal. The symmetric
solution is unique by Proposition~\ref{prop:sununiq}.
\end{proof}

Beyond the balanced case, the remaining geodesic regimes are controlled by the fixed-regime
Perron argument and the exact max formula below.

\begin{remark}\label{rem:sunrigid}
The leaves are what restore irreducibility in the short-cycle regimes. In the pure all-wrap
regime the cycle block alone is reducible for $g=3,4$, but each leaf coordinate is coupled to
the two cycle-edge coordinates at its foot, and these overlapping links connect the full
reduced system. Thus Proposition~\ref{prop:sunrigid}(ii) gives uniqueness inside every fixed
self-consistent regime without requiring a separate resonance calculation.
\end{remark}

\begin{lemma}\label{lem:fmax}
On a sun $S_{g,d}$ $(g\ge4)$, assume that the sibling leaves at each cycle vertex
$v_i$ have common weight $b_i$ (as they do in every Einstein metric by
Lemma~\ref{lem:leaf-rigidity}). A geodesic cycle edge $c_j$ of weight
$a_j\le\tfrac12L$ satisfies
\[
\kappa_{c_j}\,a_j=\max\bigl(F^B_j,\,F^W_j\bigr),\qquad
F^B_j=\tfrac1d\bigl[2a_j-s_j-(d-2)(b_j+b_{j+1})\bigr],\quad
F^W_j=F^B_j+\tfrac{2T_j-L}{d},
\]
$s_j=a_{j-1}+a_{j+1}$, $T_j=a_{j-1}+a_j+a_{j+1}$; the balanced value $F^B_j$ is taken when
$T_j\le\tfrac12L$ and the wrap value $F^W_j$ when $T_j>\tfrac12L$.
\end{lemma}
\begin{proof}
The balanced case ($T_j\le\tfrac12L$) is Proposition~\ref{prop:balanced}(i). Let $T_j>\tfrac12L$;
write $x,y$ for the endpoints of $c_j$, $x',y'$ for their outer cycle neighbours, and
$R:=L-T_j$, so $d(x',y')=\min(T_j,L-T_j)=R$. With $\varepsilon=1-\alpha$ and $e=\varepsilon/d$,
$\rho=\mu_x^\alpha-\mu_y^\alpha$ carries $\alpha-e$ at $x$, $-(\alpha-e)$ at $y$, $+e$ at $x'$ and
at each leaf of $x$, $-e$ at $y'$ and at each leaf of $y$. The coupling that moves $\alpha-e$ from
$x$ to $y$ (cost $a_j$), $e$ from $x'$ to $y'$ along the wrap (cost $R$), and $e$ from each leaf of
$x$ to the matching leaf of $y$ (cost $a_j+b_j+b_{j+1}$) has total cost
$V=(\alpha-e)a_j+eR+(d-2)e(a_j+b_j+b_{j+1})$. The potential $f(x)=a_j$, $f(y)=0$,
$f(\ell)=a_j+b_j$ ($\ell$ a leaf of $x$), $f(\ell')=-b_{j+1}$ ($\ell'$ a leaf of $y$),
$f(x')=u$, $f(y')=v$ with $u-v=R$ gives $\langle f,\rho\rangle=V$ for every
admissible $(u,v)$. Put $p=a_{j-1}$, $a=a_j$, and $q=a_{j+1}$. The core Lipschitz
conditions are equivalent to choosing
\[
 u\in[a-p,a+p]\cap[R-q,R+q],\qquad v=u-R.
\]
These intervals intersect exactly when $|a-R|\le p+q$. The two inequalities contained in
this condition are $a\le R+p+q=L-a$, i.e.\ $a_j\le L/2$, and
$R\le a+p+q$, i.e.\ $T_j\ge L/2$. Thus a choice exists in the wrap case, and the leaf
constraints follow from the core ones because the leaves are pendant. Kantorovich duality then squeezes
$V=\langle f,\rho\rangle\le W(\mu_x^\alpha,\mu_y^\alpha)\le\mathrm{cost}=V$, so $W=V$ and
$\kappa_\alpha=1-V/a_j$. Since $V$ is affine in $\varepsilon$ and $V=a_j$ at
$\varepsilon=0$, the normalized quotient $\kappa_\alpha/(1-\alpha)$ is independent of
$\alpha$ in this interval, giving
$\kappa_{c_j}a_j=(a_j-V)/\varepsilon=F^B_j+(2T_j-L)/d=F^W_j$.
\end{proof}

\begin{proposition}\label{prop:sunglobal}
For every sun $S_{g,d}$ the discrete Einstein metric is unique up to scaling among
geodesic weightings --- those in which no cycle edge exceeds $\tfrac12L$: it is the
symmetric metric of Proposition~\ref{prop:sununiq}.
\end{proposition}
\begin{proof}
The case $g=3$ is Proposition~\ref{prop:triangle-sun}, so assume $g\ge4$.
By Proposition~\ref{prop:sunrigid}(i) an Einstein metric is determined by its cycle weights
$(a_j)$ and curvature $\kappa$; set $c=\kappa d$. On a geodesic weighting each cycle edge $c_j$
lies in one of two transport regimes according as $T_j:=a_{j-1}+a_j+a_{j+1}$ is $\le$ or
$>\tfrac12L$, and $\kappa_{c_j}w(c_j)=\max(F^B_j,F^W_j)$: the through-arc value $F^B_j$, raised
to the wrap value $F^W_j=F^B_j+(2T_j-L)/d$ exactly when $T_j>\tfrac12L$ (Lemma~\ref{lem:fmax}).
If every cycle edge is balanced, the metric is symmetric by
Proposition~\ref{prop:sunbalsym}. If every cycle edge wraps, then necessarily $g\le5$; the
symmetric metric of Proposition~\ref{prop:sununiq} lies in this strict all-wrap regime, and
Proposition~\ref{prop:sunrigid}(ii) makes it the unique positive eigenvector there. It remains to
exclude a \emph{mixed} regime, with wrap-set $S$ satisfying $0<|S|<g$.

The leaf equation of Proposition~\ref{prop:sunrigid}(i) reads $a_{v-1}+a_v=(4-c)b_v$, so
$c<4$. Summing the reduced Einstein equations over the cycle edges gives
\begin{equation}\label{eq:sumT}
\sum_{j\in S}T_j=\tfrac{L}{2}\bigl(\psi(c)+|S|\bigr),\qquad
\psi(c)=c+\frac{4(d-2)}{4-c},
\end{equation}
and on $c<4$ the function $\psi$ is strictly increasing, with $\psi(c_{\mathrm B})=0$ at the
symmetric all-balanced value $c_{\mathrm B}=2-2\sqrt{d-1}$ (Theorem~\ref{thm:sun}) and
$\psi(c_{\mathrm W})=6-g$ at the symmetric all-wrap value $c_{\mathrm W}$. The wrap edges have
$T_j>\tfrac12L$, so $\sum_{j\in S}T_j>\tfrac12|S|L$ and \eqref{eq:sumT} gives $\psi(c)>0$, hence
$c>c_{\mathrm B}$; the balanced edges have $T_j\le\tfrac12L$, so
$\sum_{j\notin S}T_j\le\tfrac12(g-|S|)L$ and \eqref{eq:sumT} gives $\psi(c)\ge6-g$, hence
$c\ge c_{\mathrm W}$.

By Proposition~\ref{prop:sunrigid}(ii), $M_\sigma=4I-N_\sigma$ is nonnegative and
irreducible, and the metric is its positive Perron vector.  Applying the
Collatz--Wielandt formula to $M_\sigma$ gives
\[
 c=4-\rho(M_\sigma)=\min_{y>0}\max_i\frac{(N_\sigma y)_i}{y_i},
 \qquad
 c\le\max_i\frac{(N_\sigma y)_i}{y_i}
 \quad\text{for every }y>0,
\]
with equality only when $y$ is proportional to the Perron vector.  Here we used the
lower Collatz--Wielandt form
$\rho(M_\sigma)=\max_{y>0}\min_i(M_\sigma y)_i/y_i$ and the identity
$(M_\sigma y)_i/y_i=4-(N_\sigma y)_i/y_i$, which also fixes the direction of the
inequality.  For the two
symmetric comparison vectors, normalise the cycle coordinates to $1$ and write
\[
x_{\mathrm B}=(\mathbf 1,b_{\mathrm B}\mathbf 1),\qquad
b_{\mathrm B}=\frac{\sqrt{d-1}-1}{d-2},
\qquad
x_{\mathrm W}=(\mathbf 1,b_{\mathrm W}\mathbf 1),\qquad
b_{\mathrm W}=\frac{2}{4-c_{\mathrm W}}.
\]
Substitution in the displayed rows of $N_\sigma$ gives the complete row-ratio table
\[
\begin{array}{c|ccc}
 & B\text{-cycle row} & W\text{-cycle row} & \text{leaf row}\\ \hline
x_{\mathrm B}&c_{\mathrm B}&c_{\mathrm B}+6-g&c_{\mathrm B}\\
x_{\mathrm W}&c_{\mathrm W}-(6-g)&c_{\mathrm W}&c_{\mathrm W}.
\end{array}
\]
Indeed, the first row follows from
$-2(d-2)b_{\mathrm B}=4-2/b_{\mathrm B}=c_{\mathrm B}$; the second follows from
$4-2/b_{\mathrm W}=c_{\mathrm W}$ and
$\psi(c_{\mathrm W})=6-g$.  If $g\ge6$, choose $y=x_{\mathrm B}$.  Its largest row ratio is
$c_{\mathrm B}$, so $c\le c_{\mathrm B}$, contradicting the strict lower bound
$c>c_{\mathrm B}$ forced by the presence of a wrap edge.  This includes the boundary
case $g=6$.  If $4\le g\le5$, choose $y=x_{\mathrm W}$.  The wrap and leaf rows have
ratio $c_{\mathrm W}$, whereas every balanced row has the strictly smaller ratio
$c_{\mathrm W}-(6-g)$.  Because a mixed regime has a balanced row,
$x_{\mathrm W}$ is not the Perron vector; hence Collatz--Wielandt is strict and
$c<c_{\mathrm W}$, contradicting $c\ge c_{\mathrm W}$. Thus no mixed regime occurs, and the
symmetric metric is the only Einstein metric among geodesic weightings.
\end{proof}

It remains to treat the non-geodesic weightings, in which some cycle edge exceeds half
the perimeter --- the degenerate case set aside so far (cf.\
Proposition~\ref{prop:balanced}). They carry no Einstein metric at all.

\begin{proposition}\label{prop:nongeo}
No sun $S_{g,d}$ admits a discrete Einstein metric with a non-geodesic cycle edge.
\end{proposition}
\begin{proof}
The case $g=3$ is already excluded by Proposition~\ref{prop:triangle-sun}, whose proof
shows that every triangle-sun Einstein metric has three equal geodesic cycle edges.
Assume henceforth that $g\ge4$, and suppose that $w$ is an Einstein weighting. Define its
\emph{metric closure} on the same graph by
\[
 \widehat w_{xy}:=d_w(x,y)\qquad (xy\in E).
\]
Then $d_{\widehat w}=d_w$. Indeed, $\widehat w_e\le w_e$ gives
$d_{\widehat w}\le d_w$, while for every path $x=x_0,\ldots,x_m=y$ the triangle inequality
for $d_w$ gives
\[
 \sum_{i=0}^{m-1}\widehat w_{x_ix_{i+1}}
 =\sum_{i=0}^{m-1}d_w(x_i,x_{i+1})\ge d_w(x,y),
\]
and taking the infimum over paths gives the reverse inequality.

The adjacency relation and combinatorial degrees have not changed, so the lazy measures
\eqref{eq:measure} are the same for $w$ and $\widehat w$. Since their vertex metrics also
coincide, all Wasserstein distances and all edge curvatures coincide. Thus $\widehat w$ is
a geodesic Einstein weighting. Proposition~\ref{prop:sunglobal} makes it symmetric; write
$a>0$ for its common cycle-edge weight.

If an original cycle edge $e$ were non-geodesic, the complementary cycle arc would realise
the distance between its endpoints (a shortest path never enters a pendant leaf). Therefore
\[
 a=d_w(e)=\sum_{f\in C\setminus\{e\}}w_f
 \ge\sum_{f\in C\setminus\{e\}}\widehat w_f=(g-1)a>a,
\]
a contradiction. Hence every cycle edge of $w$ is geodesic, and
Proposition~\ref{prop:sunglobal} shows that $w$ itself is symmetric. This proves the claim.
\end{proof}

\begin{theorem}\label{thm:sununcond}
Every sun $S_{g,d}$ $(g\ge3,d\ge3)$ carries exactly one discrete Einstein metric up to
scaling among all positive weightings: the symmetric metric of
Proposition~\ref{prop:sununiq}.
\end{theorem}
\begin{proof}
For $g=3$ this is Proposition~\ref{prop:triangle-sun}. Assume $g\ge4$. Existence is
Corollary~\ref{cor:symexist}. If every cycle edge is geodesic, uniqueness is
Proposition~\ref{prop:sunglobal}. Otherwise there is a unique cycle edge of weight greater
than $L/2$, which is excluded by Proposition~\ref{prop:nongeo}.
\end{proof}

\begin{remark}\label{rem:sunglobal-scope}
The global uniqueness proof separates two issues. Proposition~\ref{prop:sunglobal} treats the
geodesic cone using the Perron-symmetry reduction of Proposition~\ref{prop:sunbalsym}, the exact
value and feasibility identities of Lemma~\ref{lem:fmax}, and Collatz--Wielandt comparison.
Proposition~\ref{prop:nongeo} then uses only metric closure and the fact that the measures
\eqref{eq:measure} depend on adjacency and combinatorial degree, not on edge length. The
non-geodesic curvature formulas recorded in the accompanying supplementary note are not used
in the proof of Theorem~\ref{thm:sununcond}.
\end{remark}

\begin{remark}\label{rem:symexist}
The small suns $S_{3,d},S_{4,d},S_{5,d}$ are Einstein but unbalanced, so decoration enlarges
the class of unbalanced Einstein graphs beyond the bare short cycles. It also restores
uniqueness: the bare $C_3$ has the entire positive orthant (modulo scaling) as Einstein
metrics at curvature $3/2$, and $C_4$ has the one-parameter family $(a,b,a,b)$ at curvature
$1$, whereas Theorem~\ref{thm:sununcond} gives a unique metric on every regular sun.
\end{remark}

\begin{remark}\label{rem:exist}
The example $C_3^+$ shows that existence can fail, while $C_4$ shows that uniqueness can
fail when the active transport operator is reducible. A structural characterization of all
unicyclic graphs admitting an Einstein metric remains open.
\end{remark}

\begin{remark}[Discrete Ricci flow]\label{rem:flow}
For the uniform-neighbour measures \eqref{eq:measure} and the weighted path metric used here,
the discrete Einstein metrics are, up to scaling, the fixed points of the normalized Ricci flow
$\dot w_e=-(\kappa_e-\bar\kappa)\,w_e$ on the edge lengths, with $\bar\kappa$ the weighted mean
curvature and $\sum_e w_e$ held fixed. Related continuous- and discrete-time Ollivier--Ricci
flows, with different normalizations, update rules, and conventions for how edge weights enter
the metric and the random-walk measures, are studied in
\cite{BLLWY,LaiBaiLin,LiMunch,NiLinGaoGu,NiLinLuoGao}. In the balanced
regime the identity $\kappa_e w_e=-(\RG w)_e$ of Theorem~\ref{thm:einstein} makes the flow
\[
   \dot w=\RG w+\bar\kappa(w)\,w .
\]
The scalar term $\bar\kappa(w)w$ enforces the normalization and does not change the
projective trajectory. Thus, for as long as the solution remains in the balanced cone,
its projective orbit agrees with that of $e^{t\RG}w(0)$. If it remains there for all
$t\ge0$, Perron--Frobenius gives convergence of the normalized weights to $w^\ast$ at
the exponential rate $\lmax-\lambda_2$. Interiority of $w^\ast$ gives a local
balanced neighbourhood (for a regular sun this holds for $g\ge7$), but does not by
itself prove that every trajectory starting in the cone remains there; a global
invariance statement requires a separate argument. Both failures of the balanced
picture are visible for the flow: on $C_3^{+}$ there is no fixed point
(Example~\ref{thm:nonexist}), and on $C_4$, where the transport-regime operator is
reducible (Section~\ref{sec:short}), the fixed points form a one-parameter family.

At the boundary girth $g=6$ the symmetric metric lies on $3a=\tfrac12L$, where the curvature
$\kappa_{c_j}w_{c_j}=\max(F^B_j,F^W_j)$ of a cycle edge is not differentiable
(Lemma~\ref{lem:fmax}), so the vector field has different one-sided projective generators.
On the balanced side the generator is $\RG$; on a wrapping side the cycle block acquires,
from the gradient of the wrap term $(2T_j-L)/d$, the
entries $\tfrac1d\bigl(1-2\cdot\mathbf 1[k\in\{j-1,j,j+1\}]\bigr)$ coupling $c_j$ to every cycle
edge $c_k$ in the projective generator $R_G^{(\sigma)}$. This coupling is absent at the symmetric metric for $g\ge7$, where
$w^\ast$ is interior and $\RG$ governs the local projective dynamics.
\end{remark}

\section*{Statement of AI Use}
During the preparation of this work, the author used Chatgpt 5.6 Sol for language and grammar
checking, assistance with checking mathematical proofs, discussion of counterexamples, and
suggestions that led to the metric-closure simplification. After using this tool, the author
reviewed and edited the content as needed and takes full responsibility for the content of
the article.

\bibliographystyle{amsplain}
\bibliography{references}

@unpublished{BCH,
  author  = {Bai, S. and Cheng, H. and Hua, B.},
  title   = {Discrete {E}instein metrics on trees},
  note    = {arXiv:2604.22449 [math.DG]},
  year    = {2026},
}

@unpublished{BCHsub,
  author  = {Bai, {S.} and Cheng, H. and Hua, B.},
  title   = {Edge subdivision and the {P}erron eigenvalue of tree {R}icci matrices},
  note    = {arXiv:2605.30949},
  year    = {2026},
}

@unpublished{BCHleaf,
  author  = {Bai, S. and Cheng, {H.} and Hua, B.},
  title   = {Spectral monotonicity under leaf attachment and limiting behavior in discrete {E}instein trees},
  note    = {arXiv:2605.23379},
  year    = {2026},
}

@unpublished{ChengHua,
  author  = {Cheng, H.},
  title   = {A classification of nonnegative-curvature discrete {E}instein metrics on trees},
  note    = {preprint},
  year    = {2026},
}

@article{LLY,
  author  = {Lin, Y. and Lu, L. and Yau, S.-T.},
  title   = {Ricci curvature of graphs},
  journal = {Tohoku Math. J.},
  volume  = {63},
  year    = {2011},
  pages   = {605--627},
}

@article{Ollivier,
  author  = {Ollivier, Y.},
  title   = {Ricci curvature of {M}arkov chains on metric spaces},
  journal = {J. Funct. Anal.},
  volume  = {256},
  year    = {2009},
  pages   = {810--864},
}

@article{LinYau,
  author  = {Lin, Y. and Yau, S.-T.},
  title   = {Ricci curvature and eigenvalue estimate on locally finite graphs},
  journal = {Math. Res. Lett.},
  volume  = {17},
  year    = {2010},
  pages   = {343--356},
}

@article{LLYflat,
  author  = {Lin, {Y.} and Lu, L. and Yau, S.-T.},
  title   = {Ricci-flat graphs with girth at least five},
  journal = {Comm. Anal. Geom.},
  volume  = {22},
  number  = {4},
  year    = {2014},
  pages   = {671--687},
  doi     = {10.4310/CAG.2014.v22.n4.a3},
  note    = {doi: 10.4310/CAG.2014.v22.n4.a3; see also the erratum by D. Cushing, R. Kangaslampi, Y. Lin, S. Liu, L. Lu, and S.-T. Yau, Comm. Anal. Geom. 29 (2021), no. 8, 1775--1781, doi: 10.4310/CAG.2021.v29.n8.a2},
}

@misc{BLYflat,
  author  = {Bai, S. and Lu, L. and Yau, S.-T.},
  title   = {Ricci-flat graphs with maximum degree at most four},
  note    = {preprint, arXiv:2103.00941},
  year    = {2021},
}

@article{CLP,
  author  = {Cushing, D. and Liu, S. and Peyerimhoff, N.},
  title   = {Bakry--{\'E}mery curvature functions on graphs},
  journal = {Canad. J. Math.},
  volume  = {72},
  year    = {2020},
  pages   = {89--143},
}

@article{CushingKamtueKoolenLiuMunchPeyerimhoff,
  author  = {Cushing, D. and Kamtue, S. and Koolen, J. and Liu, S. and M\"unch, F. and Peyerimhoff, N.},
  title   = {Rigidity of the {B}onnet--{M}yers inequality for graphs with respect to {O}llivier {R}icci curvature},
  journal = {Adv. Math.},
  volume  = {369},
  pages   = {107188},
  year    = {2020},
  doi     = {10.1016/j.aim.2020.107188},
}

@article{DevriendtLambiotte,
  author  = {Devriendt, K. and Lambiotte, R.},
  title   = {Discrete curvature on graphs from the effective resistance},
  journal = {J. Phys. Complex.},
  volume  = {3},
  year    = {2022},
  pages   = {025008},
}

@article{MW,
  author  = {M\"unch, F. and Wojciechowski, R. K.},
  title   = {Ollivier {R}icci curvature for general graph {L}aplacians: heat equation, {L}aplacian comparison, non-explosion and diameter bounds},
  journal = {Adv. Math.},
  volume  = {356},
  year    = {2019},
  pages   = {106759},
}

@article{JostLiu,
  author  = {Jost, J. and Liu, S.},
  title   = {Ollivier's {R}icci curvature, local clustering and curvature-dimension inequalities on graphs},
  journal = {Discrete Comput. Geom.},
  volume  = {51},
  year    = {2014},
  pages   = {300--322},
}

@article{Paeng,
  author  = {Paeng, S.-H.},
  title   = {Volume and diameter of a graph and {O}llivier's {R}icci curvature},
  journal = {European J. Combin.},
  volume  = {33},
  number  = {8},
  pages   = {1808--1819},
  year    = {2012},
  doi     = {10.1016/j.ejc.2012.03.029},
}

@article{BauerJostLiu,
  author  = {Bauer, F. and Jost, J. and Liu, S.},
  title   = {Ollivier-{R}icci curvature and the spectrum of the normalized graph {L}aplace operator},
  journal = {Math. Res. Lett.},
  volume  = {19},
  year    = {2012},
  pages   = {1185--1205},
}

@article{BCLMP,
  author  = {Bourne, D. P. and Cushing, D. and Liu, S. and M\"unch, F. and Peyerimhoff, N.},
  title   = {Ollivier--{R}icci idleness functions of graphs},
  journal = {SIAM J. Discrete Math.},
  volume  = {32},
  number  = {2},
  year    = {2018},
  pages   = {1408--1424},
  doi     = {10.1137/17M1134469},
}

@article{vanderHoorn,
  author  = {van der Hoorn, P. and Cunningham, W. J. and Lippner, G. and Trugenberger, C. and Krioukov, D.},
  title   = {Ollivier-{R}icci curvature convergence in random geometric graphs},
  journal = {Phys. Rev. Research},
  volume  = {3},
  year    = {2021},
  pages   = {013211},
}

@article{GosztolaiArnaudon,
  author  = {Gosztolai, A. and Arnaudon, A.},
  title   = {Unfolding the multiscale structure of networks with dynamical {O}llivier--{R}icci curvature},
  journal = {Nat. Commun.},
  volume  = {12},
  pages   = {4561},
  year    = {2021},
  doi     = {10.1038/s41467-021-24884-1},
}

@inproceedings{NiLGGS,
  author    = {Ni, C.-C. and Lin, Y.-Y. and Gao, J. and Gu, X. D. and Saucan, E.},
  title     = {Ricci curvature of the {I}nternet topology},
  booktitle = {2015 {IEEE} Conference on Computer Communications ({INFOCOM})},
  publisher = {IEEE},
  year      = {2015},
  pages     = {2758--2766},
}

@inproceedings{NiLinGaoGu,
  author    = {Ni, C.-C. and Lin, Y.-Y. and Gao, J. and Gu, X.},
  title     = {Network alignment by discrete {O}llivier--{R}icci flow},
  booktitle = {Graph Drawing and Network Visualization},
  editor    = {Biedl, T. and Kerren, A.},
  series    = {Lecture Notes in Comput. Sci.},
  volume    = {11282},
  publisher = {Springer},
  year      = {2018},
  pages     = {447--462},
  doi       = {10.1007/978-3-030-04414-5_32},
}

@article{NiLinLuoGao,
  author  = {Ni, C.-C. and Lin, Y.-Y. and Luo, F. and Gao, J.},
  title   = {Community detection on networks with {R}icci flow},
  journal = {Sci. Rep.},
  volume  = {9},
  pages   = {9984},
  year    = {2019},
  doi     = {10.1038/s41598-019-46380-9},
}

@article{LaiBaiLin,
  author  = {Lai, X. and Bai, S. and Lin, Y.},
  title   = {Normalized discrete {R}icci flow used in community detection},
  journal = {Physica A},
  volume  = {597},
  pages   = {127251},
  year    = {2022},
  doi     = {10.1016/j.physa.2022.127251},
}

@article{LiMunch,
  author  = {Li, R. and M\"unch, F.},
  title   = {The convergence and uniqueness of a discrete-time nonlinear {M}arkov chain},
  journal = {J. Funct. Anal.},
  volume  = {290},
  number  = {9},
  pages   = {111367},
  year    = {2026},
  doi     = {10.1016/j.jfa.2026.111367},
}

@article{BLLWY,
  author  = {Bai, S. and Lin, Y. and Lu, L. and Wang, Z. and Yau, S.-T.},
  title   = {Ollivier {R}icci-flow on weighted graphs},
  journal = {Amer. J. Math.},
  volume  = {146},
  number  = {6},
  pages   = {1723--1747},
  year    = {2024},
  doi     = {10.1353/ajm.2024.a944362},
}

@article{bai2021ricci,
  author  = {Bai, S. and Huang, A. and Lu, L. and Yau, S.-T.},
  title   = {On the sum of {R}icci-curvatures for weighted graphs},
  journal = {Pure Appl. Math. Q.},
  volume  = {17},
  year    = {2021},
  pages   = {1599--1617},
}

\end{document}